\documentclass{amsart}
\usepackage{graphicx}
\usepackage{color}
\newtheorem{thm}{Theorem}[section]
\newtheorem{cor}[thm]{Corollary}
\newtheorem{lem}[thm]{Lemma}
\newtheorem{prop}[thm]{Proposition}
\newtheorem{defn}[thm]{Definition}
\newtheorem{rem}[thm]{Remark}

\newcommand{\norm}[1]{\left\Vert#1\right\Vert}
\newcommand{\abs}[1]{\left\vert#1\right\vert}
\newcommand{\set}[1]{\left\{#1\right\}}
\newcommand{\Real}{\mathbb R}

\begin{document}

\title{Generalized neck analysis of harmonic maps from surfaces}

\author{Hao Yin}

\address{Hao Yin,  School of Mathematical Sciences,
University of Science and Technology of China, Hefei, China}
\email{haoyin@ustc.edu.cn }

\maketitle

\begin{abstract}
	In this paper, we study the behavior of a sequence of harmonic maps from surfaces with uniformly bounded energy on the generalized neck domain. The generalized neck domain is a union of ghost bubbles and annular neck domains, which connects non-trivial bubbles. An upper bound of the energy density is proved and we use it to study the limit of the nullity and index of the sequence.
\end{abstract}

\section{Introduction}

Let $(N,h)$ be a closed Riemannian manifold isometrically embedded in $\Real^p$ and $B$ be the unit ball of $\Real^2$. We study a sequence of harmonic maps $u_i$ from $B$ to $N$ satisfying

(U1) the energy $\int_B \abs{\nabla u_i}^2 dx$ is uniformly bounded;

(U2) for any $r>0$, $u_i$ converges smoothly to $u_\infty$ on $B\setminus B_r$ (where $B_r$ is the ball of radius $r$ centered at the origin) and 
\begin{equation}\label{eqn:concentration}
	\lim_{r\to 0} \lim_{i\to\infty} \int_{B_r} \abs{\nabla u_i}^2 dx >0.
\end{equation}
The energy concentration (as in \eqref{eqn:concentration}) leads to the existence of a sequence of pairs $(x_i,\lambda_i)$ with $x_i\to 0$ in $B$ and $\lambda_i\to 0$ such that
\[
	v_i(y)=u_i(x_i+\lambda_i y)
\]
converges to a (nontrivial) harmonic map $\omega$ from $\Real^2$ to $N$, which is known as a bubble. It is well known that there may be more than one bubbles developing at one concentration point. 

These bubbles, according to their positions $x_i$ and scales $\lambda_i$, are organized in the form of a tree. There are several expositions about the construction of the bubble tree in the literature (see \cite{ding1995energy, Parker1996}). For technical reasons, a special type of bubbles, known as ghost bubbles (in the sense that they carry no energy in the limit), is introduced as connectors in the bubble tree. We refer to Section \ref{sec:gneck} for the exact formulation. 

An edge in the tree represents an annular domain which is known as the neck. The ratio between the outer radius and the inner radius of the neck goes to infinity. While we know the limit of the sequence of scaled maps, the study of $u_i$ in the neck domain is less obvious and known as the neck analysis. The energy identity theorem and the no neck theorem imply that the energy and the oscillation of $u_i$ vanish in the neck domain. Indeed, a decay of the gradient of $u_i$ (regarded as a map on cylinder due to the conformal invariance of the problem) was proved. Recently, the author \cite{yin2019higher} proved some higher order estimate for $u_i$ in the neck domain which allows us to obtain a normal form in the center piece of the neck. 

It is well possible that the neck above is connected to a ghost bubble. It is worth emphasizing that a ghost bubble is not a real one and it serves the same purpose of connecting real bubbles (or the weak limit map $u_\infty$) as the necks. It is natural to pursue a deeper understanding of $u_i$ on the ghost bubble than the mere vanishing of energy (by its definition). This is the main topic of this paper. As a first step, we obtain an upper bound of $\abs{\nabla u_i}$. 

For a precise formulation of our main result, we need to introduce {\it the generalized neck domain}, on which our upper bound of $\abs{\nabla u_i}$ holds. It is helpful to keep the following simple case in mind, which we illustrate in the following figure. It involves a ghost bubble on top of which two real ones sit.

\begin{figure}[h]
	\centering
	\includegraphics{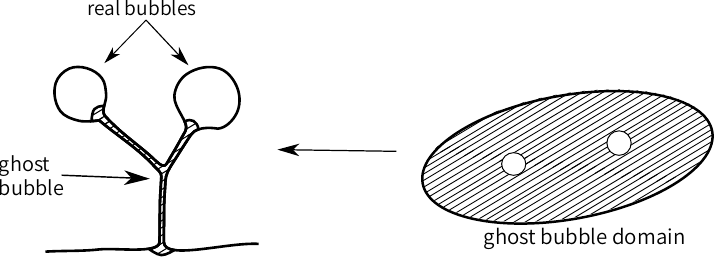}
	\caption{A simple bubble tree with ghost bubble}
	\label{fig:intro}
\end{figure}

The picture on the right shows a disk with two smaller ones removed. It is the simplest example of a so called generalized neck domain. Its image (as shown on the left) consists of a ghost bubble and three necks that are connected to it. 

In general, let's fix the sequence $u_i$. For some parameter $\delta>0$ and $l$ (real) bubbles given by $(y_i^{(j)},\lambda_i^{(j)})$ with $j=1,\cdots,l$, the generalized neck domain is (up to scaling and translation)
	\begin{equation*}
		\Omega_i=B(0, \delta) \setminus \bigcup_{j=1,\cdots,l} B(y_i^{(j)}, \delta^{-1} \lambda_i^{(j)}).
	\end{equation*}
Moreover, we assume that

(O1) the barycenter of the bubbles is the origin in the sense that
\[
	0 = \frac{1}{l} \sum_{j=1}^l y_i^{(j)};
\]

(O2) the bubbles $(y_i^{(j)},\lambda_i^{(j)})$ for $j=1,\cdots,l$ are disjoint in the sense that
\[
	B(y_i^{(j_1)},R \lambda_i^{(j_1)}) \cap B(y_i^{(j_2)}, R\lambda_i^{(j_2)})=\emptyset
\]
for any $R>0$, any $j_1\ne j_2$ and sufficiently large $i$ (depending on $R$, $j_1$ and $j_2$);

(O3) there is some $\varepsilon_0$ depending only on $N$ such that for sufficiently large $i$,
\[
	\int_{\Omega_i} \abs{\nabla u_i}^2 dx <\varepsilon_0.
\]

\begin{rem}
	(1) These assumptions arise naturally in the construction of bubble tree. This is going to be clear in Section \ref{sec:gneck}.

	(2) The assumption (O3) above is a consequence of the energy identity theorem and the definition of the ghost bubble.

	(3) In the construction in Section \ref{sec:gneck}, the generalized neck domain is going to be a translation and a scaling of the $\Omega_i$ defined above. By the nature of the problem, this does not matter.
\end{rem}

The main result of this paper is

\begin{thm}
	\label{thm:main}
	Suppose that $u_i$ is a sequence of harmonic maps satisfying (U1) and (U2) and that $\Omega_i$ is a generalized neck domain defined above, then there is some constant $C$ such that
	\[
		\abs{\nabla u_i}_{\tilde{g}_i} \leq C \qquad \text{on} \quad \Omega_i
	\]
	where $\tilde{g}_i$ is a conformal metric defined in terms of complex coordinate $z$ by
	\begin{equation}\label{eqn:tildegi}
		\tilde{g}_i := \left( 1+ \sum_{j=1}^l \frac{ (\lambda_i^{(j)})^2}{ \abs{z- x_i^{(j)}}^4} \right) dz\wedge d\bar{z} \qquad \text{on} \quad \Omega_i.
	\end{equation}
\end{thm}
\begin{rem}
	The constant $C$ in the above theorem depends not only on $N$, but also on the particular sequence. This shall be clear in the proofs. We only remark that even if $C$ is universal, this upper bound depends on the sequence in the sense that the geometry of $\Omega_i$ (hence $\tilde{g}_i$) depends on the relative position and size of the bubbles. It is also in this sense that $C$ depends on the sequence and it is inevitable, if one considers a family of such sequences which brings the tree structure to a sudden change.
\end{rem}

The upper bound of $\nabla u$ is measured with respect to the metric $\tilde{g}_i$ (see \eqref{eqn:tildegi}) defined on the multi-connected domain $\Omega_i$. This metric is related to a specific sequence of $u_i$ presented in Section \ref{sub:example}. On one hand, we feel obliged to give explicit examples to assure the readers that complicated patterns of ghost bubble domain do occur. On the other hand, the maps $u_i$ in this example are holomorphic curves (hence minimal surfaces). As if by coincidence, $\tilde{g}_i$ is the pullback metric of this family of $u_i$. Moreover, this feature of being induced metric of minimal surfaces will be useful in an application, which will be explained in a minute.

There is another way of understanding the upper bound. For that purpose, we need a different conformal metric $\bar{g}_i$. With this metric, the necks are long cylinders of radius $1$ and the ghost bubbles are multi-way connectors of uniformly bounded geometry. For any point $x\in \Omega_i$, let $d(x)$ be the distance from $x$ to the boundary to $\Omega_i$ w.r.t $\bar{g}_i$. Then the upper bound is equivalently formulated by
\[
	\abs{\nabla u_i}_{\bar{g}_i} \leq C e^{- d(x)}.
\]
Indeed, the proof of Theorem \ref{thm:main} relies on this equivalent formulation (see Theorem \ref{thm:main1p}). For the proof, we combine known techniques with some new estimate. The known ones include the three circle lemma, the ordinary differential inequality and the sharp decay estimate (see Rade \cite{rade1993decay}). The new estimate generalizes the three circle lemma (that works for annular domain) to the case of multi-connected domain (i.e. a disk with more than $2$ smaller ones removed).

As an application, we improve a result in \cite{yin2019higher}. For a harmonic map $u$ from a closed Riemannian surface $\Sigma$ to $N$, let $J_u$ be the linearization of the tension field operator $\tau(u)$, $Nul(u)$ be the dimension of its kernel and $NI(u)$ be the number of nonpositive eigenvalues of $J_u$ (counting multiplicity). A semi-continuity property of $NI$ was proved by studying the limit of eigenfunctions of $J_{u_i}$. 
\begin{thm}[Theorem 1.6 of \cite{yin2019higher}]
	\label{thm:old} Let $u_i$ be a sequence of harmonic maps from a closed Riemannian surface $\Sigma$ to a closed Riemannian manifold $N$. Let $u_\infty$ be the weak limit and $\omega_1,\cdots,\omega_l$ be all the bubbles (ghost bubbles included) in the bubble tree. Then
	\begin{equation}\label{eqn:old1}
		\limsup_{i\to \infty} NI(u_i) \leq NI(u_\infty) + \sum_{k=1}^l NI(\omega_j)
	\end{equation}
	and
	\begin{equation}\label{eqn:old2}
		\limsup_{i\to \infty} Nul(u_i) \leq Nul(u_\infty) + \sum_{k=1}^l Nul(\omega_j).
	\end{equation}
	Here in the definition of $Nul$ and $NI$ of $\omega_k$, we regard $\omega_k$ as a harmonic map from $S^2$ to $N$.
\end{thm}

The proof of Theorem \ref{thm:old} was based on an upper bound of $\abs{\nabla u_i}$ on the (annular) neck domain. If $\omega_k$ is a ghost bubble (i.e. constant map), then
\[
	Nul(\omega_k) = NI(\omega_k) = \dim N.
\]
The possible existence of ghost bubbles weakens the result of Theorem \ref{thm:old}. As an application of Theorem \ref{thm:main}, we are able to show

\begin{thm}
	\label{thm:better} Under the same assumptions as in Theorem \ref{thm:old}, \eqref{eqn:old1} and \eqref{eqn:old2} hold for $\omega_1,\cdots,\omega_l$ being the {\it real} bubbles in the bubble tree.
\end{thm}

The rest of the paper is organized as follows. In Section \ref{sec:gneck}, we recall the construction of bubble tree and define the generalized neck domain. We also define the metric $\bar{g}_i$ that is going to be used in Section \ref{sec:proof}. At the end of Section \ref{sec:gneck}, we introduce the metric $\tilde{g}_i$ in Theorem \ref{thm:main} as the pullback metric of a specific sequence of $u_i$. In Section \ref{sec:decay}, we prove key lemmas that will be used in Section \ref{sec:proof}, where Theorem \ref{thm:main} is proved. In the final section, we discuss the application and prove Theorem \ref{thm:better}.

\section*{Acknowledgement}
The author thanks Professor Yuxiang Li for numerous discussions on conformal immersion and its relation to harmonic maps. This research is supported by NSFC11971451.

\section{The generalized neck domain} \label{sec:gneck}
In the literature, there are several different ways to construct the bubble tree(see \cite{Parker1996,ding1995energy}).  In this section, we present the argument in two steps. We first obtain the set of (real) bubbles by an abstract maximizing argument. Given the set of real bubbles, we can then decompose the domain $B$ into the union of some bubble domains (one for each real bubble) and some generalized neck domains. By adding ghost bubbles, we show how the generalized neck domain is further decomposed into the union of ghost bubble domains and neck domains (i.e. annulus). Along with the decomposition, we define the metric $\bar{g}_i$.

\begin{rem}
	The material presented in this section is a re-formulation of very well known facts. Hence, the verification of elementary properties is left to the readers.
\end{rem}

\subsection{The tree of real bubbles}

\begin{defn}\label{defn:bubble}
	A (real) bubble $\mathcal B$ is a sequence of pairs $(x_i,\lambda_i)$ such that the rescaled maps $v_i(x)= u_i(x_i+\lambda_i x)$ converge weakly in $W^{1,2}$ to some {\bf nontrivial} harmonic map $\omega$ from $\Real^2$ to $N$.
\end{defn}

Intuitively, the word 'bubble' may refer to the image of the limit map $\omega$. However, Definition \ref{defn:bubble} is good for technical reasons. It dictates a region (roughly, $B(x_i,R\lambda_i)$ for large $R$) in the domain such that the maps restricted to this region converge to $\omega$. Very often, it is this region that matters.

Notice that different sequences may give the same region and lead to the same limit map  $\omega$ (up to reparametrization). Hence, any two sequences of pairs, $(x_i,\lambda_i)$ and $(y_i,\sigma_i)$ are said to be equivalent if and only if there is $c>0$ such that for all $i$,
\begin{equation*}
	c\leq \frac{\lambda_i}{\sigma_i}\leq \frac{1}{c} \qquad \text{and} \qquad \abs{x_i-y_i}\leq \frac{1}{c} \lambda_i.
\end{equation*}

\begin{rem}\label{rem:equivalent}
	Rigorously speaking, a bubble should be defined as the equivalence class of the sequence of pairs $(x_i,\lambda_i)$. However, for simplicity, we simply agree that equivalent sequences define the same bubble.
\end{rem}

It follows from the gap theorem and the total energy bound that (by passing to a subsequence) there exists a unique set of bubbles $\mathcal T$ which is maximal in the sense that one can not add another (not equivalent) real bubble. This is a consequence of finite induction.

For convenience, we add a trivial sequence $(0,1)$ ($x_i=0$ and $\lambda_i=1$) to $\mathcal T$. This 'bubble' represents the weak limit of $u_i$ and it is going to be the root of the bubble tree.

According to the size and position of the bubbles, we define a partial order which yields the tree structure in the set of real bubbles $\mathcal T$.
\begin{defn}
	(1)A bubble $\mathcal B_1=(x_i,\lambda_i)$ is said to be {\bf on top of} another bubble $\mathcal B_2=(y_i,\sigma_i)$ if and only if there is some constant $c>0$ such that
	\begin{equation*}
		\lim_{i\to \infty} \frac{\lambda_i}{\sigma_i}=0 \qquad \text{and} \qquad \abs{x_i-y_i}\leq c \sigma_i.
	\end{equation*}

	(2) $\mathcal B_1$ is said to be {\bf directly on top of} $\mathcal B_2$ if (i) $\mathcal B_1$ is on top of $\mathcal B_2$ and (ii) there is no other $\mathcal B\in \mathcal T$ satisfying $\mathcal B_1$ is on top of $\mathcal B$ and $\mathcal B$ is on top of $B_2$.
\end{defn}

By taking $\mathcal T$ as the set of vertices and taking the set of pairs $(\mathcal B_1,\mathcal B_2)$ satisfying $\mathcal B_1$ is directly on top of $\mathcal B_2$ as the set of edges, we define a graph, which is obviously a tree, and which we also denote by $\mathcal T$ for simplicity. This is the tree of real bubbles.

\subsection{The generalized neck domain}

Given the bubble tree $\mathcal T$ above, we decompose the domain $B$ into the union of bubble domains and generalized neck domains. The decomposition depends on a parameter $\delta$, which is a small positive number depending on $u_i$ and will be chosen in the constructure below.

For each bubble $\mathcal B=(x_i,\lambda_i)\in \mathcal T$, suppose that there are $l$ bubbles $\mathcal B_1,\cdots,\mathcal B_{l}$ that are directly on top of $\mathcal B$. For $j=1,\cdots,l$, let $\mathcal B_j$ be represented by the sequence $(y_i^{(j)}, \lambda_i^{(j)})$. The concentration set is defined by
\begin{equation*}
	\mathcal C:=\left\{ \lim_{i\to \infty} \frac{y_i^{(j)}-x_i}{ \lambda_i} | \quad j=1,\cdots,l \right\}.
\end{equation*}
Notice that the number of elements in the concentration set may be strictly smaller than $l$. Keeping Remark \ref{rem:equivalent} in mind, we may assume that for all $x\in \mathcal C$, we have $\abs{x}\leq 1$ by choosing a larger $\lambda_i$.

For a fixed concentration point $x\in \mathcal C$, assume that there are $l_x$ bubbles (among $\mathcal B_1$,\ldots,$\mathcal B_l$), say $\mathcal B_1,\cdots, \mathcal B_{l_x}$, satisfying
\begin{equation*}
	\lim_{i\to \infty} \frac{y_i^{(j)}-x_i}{\lambda_i}=x,\qquad j=1,\cdots,l_x.
\end{equation*}
These $\mathcal B_{1},\cdots,\mathcal B_{l_x}$ are said to {\bf be concentrated at} $x$.

At each concentration point $x\in \mathcal C$, we define a center of mass position
\begin{equation*}
	c_i= \frac{1}{l_x}\sum_{j=1}^{l_x} y_i^{(j)}.
\end{equation*}
Obviously, $\lim_{i\to\infty} \frac{c_i-x_i}{\lambda_i}=x$. We write $c_i(x)$ if we need to emphasize its dependence on $x\in \mathcal C$.

\begin{defn}
The bubble domain of $\mathcal B$ is
\begin{equation*}
	\Omega_{\mathcal B}= B(x_i, \delta^{-1} \lambda_i) \setminus \bigcup_{x\in \mathcal C} B(c_i(x), \delta \lambda_i).
\end{equation*}
\end{defn}
Here we assume that $\delta$ is small so that the balls $B(c_i(x),\delta \lambda_i)$ for $x\in \mathcal C$ are disjoint. In case that $\mathcal C$ is empty set, namely, there is no bubble on top of $\mathcal B$, the bubble domain is just $B(x_i,\delta^{-1}\lambda_i)$.

In case that $\mathcal C$ is not empty, we define 
\begin{defn}\label{defn:gnd}
	The generalized neck domain at $x$ is 
	\begin{equation*}
		\Omega_{\mathcal B,x}=B(c_i, \delta\lambda_i) \setminus \bigcup_{j=1,\cdots,l_x} B(y_i^{(j)}, \delta^{-1} \lambda_i^{(j)}).
	\end{equation*}
\end{defn}
Here we omit the routine verification that the balls $B(y_i^{(j)}, \delta^{-1}\lambda_i^{(j)})$ are disjoint and contained in $B(c_i,\delta\lambda_i)$ when $i$ is sufficiently large. 

The topology of generalized neck domain is a disk with finitely many smaller ones removed. Notice that we have no control over the size and position of these removed disks. It is the main task of this paper to study the behavior of $u_i$ in this domain.


\subsection{Decomposition of generalized neck domain by adding ghost bubbles}\label{sub:gnd}
Let $x$ be a concentration point of the bubble $\mathcal B$ as in the previous subsection. When $l_x=1$, we have $c_i=y_i^{(1)}$ and the generalized neck domain $\Omega_{\mathcal B,x}$ takes the particular simple form
\begin{equation*}
	B(y_i^{(1)}, \delta \lambda_i)\setminus B(y_i^{(1)}, \delta^{-1} \lambda_i^{(1)}).
\end{equation*}
Such annulus type domain is called {\bf simple neck domain} and known results on the neck analysis apply to this type of neck domain. In this simple case, we define
\begin{equation}
	w(z)=\frac{1}{\abs{z-c_i}^2}\qquad \text{on} \quad B(c_i,\delta\lambda_i)\setminus B(c_i,\delta^{-1}\lambda_i^{(1)}).
	\label{eqn:ww1}
\end{equation}
This function $w$ is going to be used as a conformal factor in the definition of $\bar{g}_i$.

When $l_x>1$, we describe below an induction process, which by adding some more ghost bubbles, further decompose the generalized neck domain into the union of simple neck domains and (small) ghost bubble domains.

By taking subsequence, we may assume that
\begin{equation}\label{eqn:sigmai}
	\sigma_i:=2 \max_{j_1,j_2=1,\cdots,l_x} \abs{y_i^{(j_1)}-y_i^{(j_2)}}=\abs{y_i^{(1)}-y_i^{(2)}} >0.
\end{equation}
The sequence $(c_i,\sigma_i)$ (up to equivalence as in Remark \ref{rem:equivalent}) represents a {\bf ghost bubble} $\tilde{\mathcal B}$. Due to the maximality of $\mathcal T$, one can check that the limit of
\begin{equation*}
	v_i(y)= u_i(c_i + y \sigma_i)
\end{equation*}
is constant map. 
Indeed, we have

(1) $(c_i,\sigma_i)$ is not equivalent to any bubble in $\mathcal T$;

(2) $B(c_i, \delta\lambda_i)\setminus B(c_i,2\sigma_i)$ is a simple neck domain, on which we define
\begin{equation}
	w(z)= \frac{1}{ \abs{z-c_i}^2};
	\label{eqn:ww2}
\end{equation}

(3) ${\mathcal B}_1,\cdots,{\mathcal B}_{l_x}$ are directly on top of $\tilde{\mathcal B}$;

(4) the concentration set 
\[
	\tilde{\mathcal C}=\set{\lim_{i\to \infty} \frac{y_i^{(j)}-c_i}{\sigma_i}|\, j=1,\cdots,l_x}
\]
contains at least two points and is a subset of $B(0,1/2)$ (see \eqref{eqn:sigmai}).

By choosing $\delta$ small, we may assume that the minimal distance between any two points in $\tilde{\mathcal C}$ is larger than $3\delta$. For each $y\in \tilde{\mathcal C}$, define the center of mass (as before)
\begin{equation*}
	c_i(y)= \frac{1}{l_y} \sum_{j=1}^{l_y} y_i^{(\alpha_j)}
\end{equation*}
where $(y_i^{(\alpha_j)}, \lambda_i^{(\alpha_j)}) (j=1,\cdots,l_y)$ are a choice of $l_y$ bubbles among $\mathcal B_1,\cdots,\mathcal B_{l_x}$.

\begin{defn}
	\label{defn:gbd} The ghost bubble domain is defined to be
	\begin{equation*}
		B(c_i, 2\sigma_i) \setminus \bigcup_{y\in \tilde{\mathcal C}} B(c_i(y), \delta \sigma_i).
	\end{equation*}
\end{defn}
On the above ghost bubble domain, we choose $w$ to be any smooth functions satisfying 

(W1) there is some constant $C>0$ such that
\begin{equation*}
	\frac{\sigma_i^2}{C}\leq w(z)\leq C\sigma_i^2;
\end{equation*}

(W2) $w$ is $\frac{1}{\abs{z-c_i}^2}$ in a neighborhood of $\partial B(c_i, 2\sigma_i)$;

(W2) $w$ is $\frac{1}{\abs{z-c_i(y)}^2}$ in a neighborhood of $\partial B(c_i(y), \delta\sigma_i)$.

For each $y\in \tilde{\mathcal C}$, it is a concentration point on the ghost bubble $(c_i,\sigma_i)$. We repeat the construction above. Notice that the total number of real bubbles (directly on top) concentrated at $y$ becomes strictly smaller than $l_x$. Hence the induction stops after finitely many steps.

We conclude this subsection by setting
\begin{equation}
	\bar{g}_i= w(z) dz\wedge d\bar{z}, \qquad \text{on} \quad \Omega_{\mathcal B,x}
	\label{eqn:bargi}
\end{equation}
where $w$ is defined by \eqref{eqn:ww1}, \eqref{eqn:ww2} and (W1-W3).
%
%
\subsection{An example of bubble tree}\label{sub:example}

This short section consists of two parts. The first part is an example which demonstrates that ghost bubbles and very complicated generalized neck domains do occur. The second part shows that the metric in \eqref{eqn:tildegi} is the pullback of some holomorphic maps into $\mathbb C^n$. This not only helps the understanding of Theorem \ref{thm:main}, but also plays a role in the proof of Theorem \ref{thm:better}.

Since this paper deals with ghost bubbles, it is natural to ask whether there exists a sequence of $u_i$ as in Theorem \ref{thm:main} that the construction in the previous subsections leads to a ghost bubble. Further more, is there a generalized neck domain as constructed above such that the punctured disks shrink and approach each other at arbitrary speed?

Indeed, the following example shows that one can prescribe the number of bubbles, the position and the scale of each (real) bubble, so that the argument in Section \ref{sub:gnd} gives a generalized neck domain as complicated as one needs. 

Precisely, let $l$ be the number of (real) bubbles. Let $(x_i^{(j)},\lambda_i^{(j)})$ be the center and the scale of the $j$-th bubble. In terms of the holomorphic coordinate $z$ on $\mathbb C$, we define a sequence of maps $\tilde{u}_i$ from $B$ to $\mathbb C^{l+1}$ to be
\begin{equation}\label{eqn:defui}
	\tilde{u}_i(z)=\left( z, \frac{\lambda_i^{(1)}}{z-x_i^{(1)}},\cdots,\frac{\lambda_i^{(l)}}{z-x_i^{(l)}} \right).	
\end{equation}
Using the homogeneous coordinates $[z_1,z_2]$ of $\mathbb CP^1$, we may regard $\mathbb C$ as an open subset of $\mathbb CP^1$, via the identification,
\begin{equation*}
	z\mapsto [z,1].
\end{equation*}
Similarly, $\mathbb C^{l+1}$ is identified with an open subset of $\mathbb CP^{l+1}$ via
\begin{equation*}
	(z_1,\cdots,z_{l+1})\mapsto [z_1,\cdots,z_{l+1},1].
\end{equation*}
Hence $\tilde{u}_i$ can be extended in a unique way as a map $u_i$ from $\mathbb CP^1$($S^2$) to $\mathbb CP^{l+1}$ and the map $u_i$ is holomorphic (hence harmonic). It is elementary to check that there are exactly $l$ bubbles occurring at the prescribed position and rate.

Next, we assume that

(E1) the $l$ bubbles concentrate at $0$, i.e.
\[
	\lim_{i\to \infty} x_i^{(j)}=0,\qquad j=1,\cdots,l.
\]

(E2) the center of mass is $0$, i.e.
\[
	\sum_{j=1}^l x_i^{(j)}=0.
\]

(E3) the $l$ bubbles are separated from each other, in other words, no one is on top of another. That is, for any $R>0$ and $j_1\ne j_2$,
\[
	B(x_i^{(j_1)},R \lambda_i^{j_1}) \cap B(x_i^{(j_2)}, R\lambda_i^{(j_2)})=\emptyset
\]
for sufficiently large $i$.

The generalized neck domain given in Definition \ref{defn:gnd} is
\[
	\Omega = B(0,\delta) \setminus \bigcup_{j=1,\ldots,l} B(x_i^{(j)},\delta^{-1}\lambda_i^{(j)}).
\]
For the $\tilde{u}_i$ in \eqref{eqn:defui}, we regard it as a holomorphic map from $\Omega$ to $\mathbb C^{l+1}$ and the pullback metric is
\begin{equation}
	\tilde{g}_i:= \left( 1+ \sum_{j=1,\ldots,l} \frac{(\lambda_i^{(j)})^2}{ \abs{z-x_i^{(j)}}^4} \right) dz\wedge d\bar{z}. 
	\label{eqn:metric}
\end{equation}
To conclude this section, we remark that $\tilde{u}_i$ parametrizes an embedded minimal surface in $\mathbb C^{l+1}$. The surface has $l+1$ ends and its tangent cone at the infinity is the union of $l+1$ coordinate planes of $\mathbb C^{l+1}$.

\section{Various energy decay estimates}\label{sec:decay}
In this section, we prove two estimates about the decay/growth of energy density on different domains. The first one is a generalization of the well known three circle lemma. The second one is a sharp growth estimate on long cylinder. 

Before we start, we note the following convention about the notations. We will use $\tilde{c}$ for universal constants, $c$ for constants that depend on the target manifold $N$, and $C$ for constants that depend both on $N$ and the particular sequence of maps $u_i$. In general, these constants may vary from line to line. However, subscripts will be added, if it is necessary to note the distinction between them.

Moreover, throughout this section, there is a small constant $\varepsilon_1$ appearing in the assumptions of the following results. We remark that it depends only on $N$, not on the sequence $u_i$. 

\subsection{Generalized three circle estimate}\label{sub:generalized}

The application of the three circle estimate to the study of harmonic maps has a long tradition(see \cite{simon1983asymptotics, QT1997, LY2016,ai2017neck}). We start by recalling the following well known result.
\begin{lem}
	\label{lem:3c} There is some $\varepsilon_1(N)>0$. For any $\beta>0$, there is $L_0>1$ (depending only on $\beta$) such the following is true for any $L>L_0$. 
	Assume that $u$ is a harmonic map defined on $[0,3L]\times S^1$ satisfying
	\begin{equation}\label{eqn:DT3c}
		\int_{[t,t+1]\times S^1} \abs{\nabla u}^2 \leq \varepsilon_1,\qquad \forall t\in [0,3L-1]
	\end{equation}
	and
\begin{equation}\label{eqn:poho}
	\int_{\set{t}\times S^1} \abs{\partial_\theta u}^2 - \abs{\partial_t u}^2 =0,\qquad \forall t\in [0,3L].
\end{equation}
	Then
	\begin{equation}\label{eqn:convex}
		\int_{[L,2L]\times S^1} \abs{\nabla u}^2 \leq \beta \left( \int_{[0,L]\times S^1} + \int_{[L,3L]\times S^1} \right) \abs{\nabla u}^2.
	\end{equation}
\end{lem}

Lemma \ref{lem:3c} compares the energy on a piece of cylinder with the energy on adjacent pieces of the same length and concludes that (when the assumptions hold) at least one of the two pieces have (significantly) larger energy.

In this section, we prove a generalization of this fact. We compare the energy on a ghost bubble domain with the energy on cylinders that are directly connected to it. To be precise, we recall that for a sequence of harmonic maps $u_i$, the ghost bubble domain in Definition \ref{defn:gbd} is
\[
	B(y_i,2\sigma_i) \setminus \bigcup_{y\in \tilde{\mathcal C}} B(c_i(y), \delta \sigma_i),
\]
where $\tilde{\mathcal C}$ is the energy concentration set of the scaled sequence $v_i(y)=u_i(y_i + y \sigma_i)$. Since the problem under investigation is scaling invariant, we may assume that the ghost bubble domain is
\[
	B(y_i,2) \setminus \bigcup_{y\in \tilde{\mathcal C}} B(c_i(y), \delta),
\]
and
\begin{equation*}
	\lim_{i\to \infty} y_i =0.
\end{equation*}
This domain varies with $i$. However, the number of points in $\mathcal C$ is bounded, their distances to the origin are bounded and the distance between any pair is bounded from below. Hence, by passing to a subsequence if necessary, the ghost bubble domain approaches a limit,
\[
	\Omega_0= B(0,2) \setminus \bigcup_{y\in \tilde{\mathcal C}} B(y, \delta).
\]
For some small $\eta>0$ to be determined later, we set
\begin{equation}\label{eqn:omegaj}
	\Omega_j= B(0, 2 \eta^{-j}) \setminus \bigcup_{y\in \tilde{\mathcal C}} B(y,\delta \eta^j).
\end{equation}

In what follows, we prove estimates for the energy of harmonic map $u$ defined on $\Omega_j$ (see Lemma \ref{lem:upperbound} and Lemma \ref{lem:4d}). The constants appeared there depend on $\Omega_0$, or more precisely, depend on $\tilde{\mathcal C}$ and $\delta$, which in turn depend on the particular sequence. These estimates hold for the original sequence $u_i$ for sufficiently large $i$ with the same set of constants. It is in this sense that we say the estimates depend on the sequence $u_i$.

The following is a set of natural assumptions under which our estimates hold and they are verified easily in the construction of bubble tree.

(S1) 
\begin{equation*}
	\int_{\Omega_0} \abs{\nabla u}^2 \leq \varepsilon_1;
\end{equation*}

(S2) for any $y\in \tilde{\mathcal C}$ and $\rho\in (\delta \eta^3,\delta/2)$,
\begin{equation*}
	\int_{B(y,2\rho)\setminus B(y,\rho)} \abs{\nabla u}^2 \leq \varepsilon_1;
\end{equation*}

(S3) for any $\eta^{-3}\geq \rho\geq 2$,
\begin{equation*}
	\int_{B(0,2\rho)\setminus B(0,\rho)} \abs{\nabla u}^2 \leq \varepsilon_1.
\end{equation*}

Our first result is the following lemma.

\begin{lem}
	\label{lem:upperbound}
	Let $u$ be a harmonic map defined on $\Omega_3$ satisfying (S1)-(S3) for some small $\varepsilon_1$ depending only on $N$. There is a constant $C_1$ depending on $N$ and $\Omega_0$ but {\bf not} on $\eta$ such that
	\begin{equation*}
		\int_{\Omega_1} \abs{\nabla u}^2 \leq C_1 \int_{\Omega_2\setminus \Omega_1} \abs{\nabla u}^2.
	\end{equation*}
\end{lem}

\begin{proof}
	Assume that the lemma is false. Then, there is a sequence of $\eta_i>0$ and a sequence of harmonic maps $u_i$ satisfying (S1)-(S3) such that
	\begin{equation}\label{eqn:nottrue}
		\int_{\Omega_1(\eta_i)} \abs{\nabla u_i}^2 \geq i \int_{\Omega_2(\eta_i)\setminus \Omega_1(\eta_i)} \abs{\nabla u_i}^2.
	\end{equation}
	\begin{rem}
	(1) We have used the notation $\Omega_1(\eta_i)$ and $\Omega_2(\eta_i)$ to emphasize the dependence on $\eta_i$. When this dependence is clear from the context, we simply write $\Omega_1$ and $\Omega_2$.

	(2) The sequence $u_i$ is not the sequence in the main theorem. We recycle the notation for simplicity and this usage is valid only in this proof.
	\end{rem}

To get a contradiction, we distinguish two cases.

{\bf Case 1:} $\liminf_{i\to \infty}\eta_i>0$. By passing to a subsequence, we assume that $\eta_i\to \eta$. 

	The $\varepsilon$-regularity theorem of harmonic maps and (S1)-(S3) together imply the existence of a smooth limit $u_\infty$ of $u_i$ defined on $\Omega_2(\eta)$, which is also a harmonic map.
	If $\int_{\Omega_1(\eta_i)} \abs{\nabla u_i}^2$ has a positive lower bound, $u_\infty$ is nontrivial. However, \eqref{eqn:nottrue} implies that $u_\infty$ is constant map on $\Omega_2\setminus \Omega_1$. This is a contradiction to the unique continuation theorem(\cite{sampson1978some}).

	If $\int_{\Omega_1} \abs{\nabla u_i}^2\to 0$, then we scale the ambient space $\Real^p$ in which $N$ is embedded and set 
	\begin{equation}\label{eqn:scaleui}
		\tilde{u}_i= \frac{\varepsilon_1^{1/2} u_i}{\left( \int_{\Omega_1} \abs{\nabla u_i}^2 \right)^{1/2}}.
	\end{equation}
	After the scaling, we have
	\begin{equation}\label{eqn:tildeui}
		\int_{\Omega_1} \abs{\nabla \tilde{u}_i}^2 =\varepsilon_1,
	\end{equation}
	which together with \eqref{eqn:nottrue} implies	
	\begin{equation}
		\int_{\Omega_2} \abs{\nabla \tilde{u}_i}^2 \leq 2\varepsilon_1.
		\label{eqn:uibound}
	\end{equation}
	Notice that $\tilde{u}_i$ is now a harmonic map into a different target manifold $N_i$, which converges to a linear subspace of $\Real^p$ as $i\to \infty$. Since the small constant in the $\varepsilon$-regularity theorem is uniform for all $N_i$, \eqref{eqn:uibound} provides the uniform estimate that yields a limit $\tilde{u}_\infty$ defined on $\Omega_2$. Due to the scaling, $\tilde{u}_\infty$ is a harmonic function.
	By \eqref{eqn:tildeui}, the limit $\tilde{u}_\infty$ is not trivial.
	However, its restriction to $\Omega_2\setminus \Omega_1$ is constant. This is impossible and we get a contradiction.

	{\bf Case 2:} $\eta_i\to 0$.
	We define $\tilde{u}_i$ as in \eqref{eqn:scaleui}. The same argument as above gives a limit $\tilde{u}_\infty$, which is a harmonic map if $\int_{\Omega_1} \abs{\nabla u_i}^2$ has a positive lower bound and is a harmonic function if otherwise. Since $\eta_i\to 0$, the domain $\Omega_1(\eta_i)$ converges to $\Real^2 \setminus \tilde{\mathcal C}$ and
	\begin{equation*}
		\int_{\Real^2 \setminus \tilde{\mathcal C}} \abs{\nabla \tilde{u}_\infty}^2 \leq \varepsilon_1.
	\end{equation*}

	Due to the removable singularity theorem and the gap theorem of harmonic map, or the fact that there is no nontrivial harmonic function on $\Real^2$ with bounded Dirichlet energy, $\tilde{u}_\infty$ must be constant map/function (if $\varepsilon_1$ is small). To get a contradiction, it suffices to prove
\begin{equation}\label{eqn:suffice}
	\int_{B(0,4)\setminus \left( \bigcup_{y\in \tilde{\mathcal C}} B(y,\delta/2) \right)} \abs{\nabla \tilde{u}_i}^2 \geq \frac{1}{\tilde{c}_1} \varepsilon_1.
\end{equation}
Here $\tilde{c}_1$ is a universal constant that will be made clear in a minute.

If \eqref{eqn:suffice} is not true, 
\begin{equation}\label{eqn:oneend}
	\int_{(B(0,4)\setminus B(0,2)) \cup\left( \bigcup_{y\in \tilde{\mathcal C}} B(y,\delta)\setminus B(y,\delta/2) \right) } \abs{\nabla \tilde{u}_i}^2 
	\leq \frac{1}{\tilde{c}_1}\varepsilon_1.
\end{equation}

By \eqref{eqn:nottrue} and \eqref{eqn:tildeui}, we have
\begin{equation*}
	\varepsilon_1 =\int_{\Omega_1(\eta_i)} \abs{\nabla \tilde{u}_i}^2 \geq i \int_{\Omega_2(\eta_i)\setminus \Omega_1(\eta_i)} \abs{\nabla \tilde{u}_i}^2,
\end{equation*}
which implies that for $i>\tilde{c}_1$,
\begin{equation}\label{eqn:theother}
	\int_{B(0,\eta_i^{-2})\setminus B(0,\eta_i^{-2}/2)\cup \left( \bigcup_{y\in \tilde{\mathcal C}} B(y,2\eta^2_i \delta) \setminus B(y,\eta^2_i \delta) \right)} \abs{\nabla \tilde{u}_i}^2 \leq \frac{1}{\tilde{c}_1}\varepsilon_1.
\end{equation}
If $m$ is the number of points in $\tilde{\mathcal C}$, we now have $m+1$ annular domains,
\[
	B(0,\eta_i^{-2})\setminus B(0,2)\quad  \mbox{and} \quad B(y,\delta)\setminus B(y,\eta_i^2\delta) \quad \mbox{for each}\, y\in \tilde{\mathcal C}.
\]
As in Case 1, we still have \eqref{eqn:uibound}, which allows us to apply the energy identity theorem to these $m+1$ annular domains simultaneously. Due to \eqref{eqn:oneend} and \eqref{eqn:theother}, by choosing $\tilde{c}_1$ large, we can have
\begin{equation*}
	\int_{(B(0,\eta_i^{-1})\setminus B(0,4)) \cup \left( \bigcup_{y\in \tilde{\mathcal C}} B(y,\delta/2) \setminus B(y, \eta_i \delta) \right)} \abs{\nabla \tilde{u}_i}^2 < \frac{1}{2}\varepsilon_1.
\end{equation*}
Since we may choose $\tilde{c}_1>2$, the above inequality and the assumed falsity of \eqref{eqn:suffice} imply that
\begin{equation*}
	\int_{\Omega_1} \abs{\nabla \tilde{u}_i}^2 <\varepsilon_1,
\end{equation*}
which is a contradiction to \eqref{eqn:tildeui}.
\end{proof}

An unfavorable aspect of the above lemma is that we have no control on the size of the constant $C_1$, because the proof is by contradiction. On the other hand, $C_1$ does not depend on $\eta$. By choosing $\eta$ small, we obtain the following counterpart of Lemma \ref{lem:3c}.

\begin{lem}
	\label{lem:4d}
	Suppose that $u$ satisfies (S1)-(S3). For any $\beta>0$, there is $\eta_0>0$ small such that for all $\eta<\eta_0$,
	\begin{equation*}
		\int_{\Omega_2(\eta)} \abs{\nabla u}^2 \leq \beta \int_{\Omega_3(\eta)\setminus \Omega_2(\eta)} \abs{\nabla u}^2.
	\end{equation*}
\end{lem}

\begin{proof}
	By Lemma \ref{lem:upperbound},
	\begin{equation}\label{eqn:well3}
		\int_{\Omega_1\setminus \Omega_0} \abs{\nabla u}^2 \leq C_1 \int_{\Omega_2\setminus \Omega_1} \abs{\nabla u}^2.
	\end{equation}

	For $\tilde{\beta}$ to be determined, Lemma \ref{lem:3c} implies the existence of some $\eta$ such that
	\begin{equation}\label{eqn:well4}
		\int_{\Omega_2\setminus \Omega_1} \abs{\nabla u}^2 \leq \tilde{\beta} \left( \int_{\Omega_1\setminus \Omega_0} \abs{\nabla u}^2 + \int_{\Omega_3\setminus \Omega_2} \abs{\nabla u}^2 \right).	
	\end{equation}
	By \eqref{eqn:well3} and \eqref{eqn:well4}, we have
	\begin{equation*}
		\int_{\Omega_2\setminus \Omega_1} \abs{\nabla u}^2 \leq \frac{\tilde{\beta}}{1-C_1\tilde{\beta}} \int_{\Omega_3\setminus \Omega_2} \abs{\nabla u}^2,
	\end{equation*}
	which implies (using Lemma \ref{lem:upperbound} again)
	\begin{equation*}
		\int_{\Omega_2} \abs{\nabla u}^2 \leq \frac{(1+C_1)\tilde{\beta}}{1- C_1\tilde{\beta}} \int_{\Omega_3\setminus \Omega_2} \abs{\nabla u}^2.
	\end{equation*}
	It suffices to choose $\tilde{\beta}$ small so that
	\begin{equation*}
		\frac{(1+C_1)\tilde{\beta}}{1-C_1\tilde{\beta}}<\beta.
	\end{equation*}
	Taking $\tilde{\beta}$ as the $\beta$ is Lemma \ref{lem:3c}, we obtain an $\eta$ such that the above computation works.
\end{proof}

\subsection{Optimal decay estimate}\label{sub:optimal}
In this section, we are interested in a long cylinder $[0,\tilde{L}]\times S^1$. Assume that $\tilde{L}$ is a multiple of some $L>2$ such that
\begin{equation*}
	[0,\tilde{L}]\times S^1= \bigcup_{i=1}^m W_i,
\end{equation*}
where $W_i= [(i-1)L, iL]\times S^1$. 

The aim of this section is to prove the following.
\begin{lem}
	\label{lem:mild} There exists some $\varepsilon_1(N)>0$. If $u$ is a harmonic map from $[0,\tilde{L}]\times S^1$ to $N$ satisfying that
	\begin{equation}
		\int_{W_i} \abs{\nabla u}^2 < \varepsilon_1(N),\qquad \int_{\set{t}\times S^1} \abs{\partial_t u}^2 - \abs{\partial_\theta u}^2 =0
		\label{eqn:smalle}
	\end{equation}
	and
	\begin{equation}
		\int_{W_{i}} \abs{\nabla u}^2 \leq \frac{1}{2} \int_{W_{i+1}} \abs{\nabla u}^2, \qquad \forall i=1,2,\cdots,m-1,
		\label{eqn:growth}
	\end{equation}
	then
	\begin{equation}
		\int_{W_{1}} \abs{\nabla u}^2 \leq C(L)e^{-2 \tilde{L}} \int_{W_m} \abs{\nabla u}^2.
		\label{eqn:optimal}
	\end{equation}
\end{lem}

By some simple arguments, we justify some further assumptions that helps in the proof.

(A1) Obviously, by \eqref{eqn:growth}, it is enough to show 
	\begin{equation}
		\int_{W_{2}} \abs{\nabla u}^2 \leq C(L)e^{- 2 \tilde{L}} \int_{W_m} \abs{\nabla u}^2.
		\label{eqn:optimal2}
	\end{equation}
	Due to \eqref{eqn:growth} again, this is trivial if $\tilde{L}\leq 5L$. Hence, we can argue by induction and assume that \eqref{eqn:optimal2} is proved for $\tilde{L}=(m-1)L$. Notice that the constant $C(L)$ in \eqref{eqn:optimal2} should not depend on $m$. We may assume further that
	\begin{equation}\label{eqn:ass1}
		\int_{W_3} \abs{\nabla u}^2 \leq e^{2L} \int_{W_2} \abs{\nabla u}^2.
	\end{equation}
	If otherwise, we may apply the induction hypothesis to the cylinder $[L,\tilde{L}]\times S^1$ to see
\[
\int_{W_3}\abs{\nabla u}^2 \leq C(L)e^{-2 (\tilde{L}-L)} \int_{W_m} \abs{\nabla u}^2, 
\]
from which \eqref{eqn:optimal2} follows. Using the elliptic estimate, \eqref{eqn:growth} and \eqref{eqn:ass1}, we have
\begin{equation}\label{eqn:goodleft}
	\sup_{t\in [L,2L]} \sup_{[t,t+1]\times S^1} \abs{\partial_\theta^2 u}^2 \leq C_1(L) \int_{W_2} \abs{\nabla u}^2.
\end{equation}

(A2) Near the other end of the cylinder, we consider a natural number $m_1$ such that $m-m_1$ is bounded by a constant depending on $L$ and
\begin{equation}
	\label{eqn:ass_up}
	\int_{W_{m_1}} \abs{\nabla u}^2 \leq \frac{\varepsilon_1}{C_1(L)\cdot L}.
\end{equation}

(A3) By a similar argument as in (A1), we may assume that
\begin{equation}
	\int_{W_{m_1}} \abs{\nabla u}^2 \leq e^{2L} \int_{W_{m_1-1}} \abs{\nabla u}^2.
	\label{eqn:ass_right}
\end{equation}
Together with \eqref{eqn:growth}, it implies that 
\begin{equation}\label{eqn:goodright}
	\sup_{t\in [(m_1-2)L,(m_1-1)L]} \sup_{[t-1,t]\times S^1} \abs{\partial_\theta^2 u}^2 \leq C_1(L) \int_{W_{m_1-1}} \abs{\nabla u}^2.
\end{equation}

For the proof of Lemma \ref{lem:mild}, it suffices to show
\begin{equation}
	\int_{W_2} \abs{\nabla u}^2 \leq C(L) e^{-2 m_1L} \int_{W_{m_1-1}} \abs{\nabla u}^2.
	\label{eqn:optimal3}
\end{equation}
By the mean value theorem, there is $t_i\in [(i-1)L,iL]$ such that
\begin{equation*}
	\int_{W_i} \abs{\nabla u}^2 = L\cdot \int_{\set{t_i}\times S^1} \abs{\nabla u}^2.
\end{equation*}
Hence, finally, the proof of Lemma \ref{lem:mild} is reduced to proving
\begin{equation}
	\int_{\set{t_2}\times S^1} \abs{\nabla u}^2 \leq C(L) e^{-2(t_{m_1-1}-t_2)} \int_{\set{t_{m_1-1}}\times S^1} \abs{\nabla u}^2.
	\label{eqn:final}
\end{equation}
(A1-A3) above implies that $u|_{[t_2,t_{m_1-1}]\times S^1}$ satsfies the assumption of the following proposition with $C_2(L)=(C_1(L)\cdot L)^{1/2}$.

\begin{prop}\label{prop:sharp} There is some $\varepsilon_1(N)>0$. Assume that $u$ is a harmonic map defined on $[0,T]\times S^1$ satisfying
	\begin{equation}\label{eqn:DT1}
		\sup_{\set{t}\times S^1} \abs{\nabla u} + \abs{\nabla^2 u} \leq \varepsilon_1^{1/2},\qquad \forall t\in [0,T]
	\end{equation}
	and
	\begin{equation}\label{eqn:poho1}
		\int_{\set{t}\times S^1} \abs{\partial_\theta u}^2 - \abs{\partial_t u}^2 =0,\qquad \forall t\in [0,T].
\end{equation}
	If
	\begin{equation*}
		\left(\int_{\set{0}\times S^1} \abs{\nabla u}^2\right)^{1/2} = a \quad \text{and} \quad \left(\int_{\set{T}\times S^1} \abs{\nabla u}^2\right)^{1/2} = b
	\end{equation*}
	and
	\begin{equation}
		\label{eqn:difficult}
		\sup_{[0,1]\times S^1} \abs{\partial_\theta^2 u}\leq C_2(L) a< \varepsilon_1^{1/2}, \quad  \sup_{[T-1,T]\times S^1} \abs{\partial_\theta^2 u}\leq C_2(L) b \leq \varepsilon_1^{1/2},
	\end{equation}
	then for any $t\in [0,T]$,
	\begin{equation}
		\left( \int_{ \set{t}\times S^1} \abs{\nabla u}^2\right)^{1/2} \leq 2 E_1 e^{- t} + 2 E_2 e^{t}
		\label{eqn:sharp}
	\end{equation}
	where
	\begin{equation*}
		E_1=\frac{ae^{2T}-be^T}{e^{2T}-1} \quad \text{and} \quad E_2=\frac{be^T-a}{e^{2T}-1}.
	\end{equation*}
\end{prop}
\begin{rem}
Notice that the right hand side of \eqref{eqn:sharp} is the solution of the ODE
\begin{equation*}
	g''=g \qquad \text{with} \quad g(0)=2a, \quad g(T)=2b.
\end{equation*}
\end{rem}

\begin{proof}
	Due to \eqref{eqn:poho1}, we set
	\begin{equation*}
		f(t)= \frac{1}{2} \int_{\set{t}\times S^1} \abs{\nabla u}^2 = \int_{\set{t}\times S^1} \abs{\partial_\theta u}^2.
	\end{equation*}
	A computation (following the Lemma 2.1 of \cite{LW1998}) yields
	\begin{equation}
		\begin{split}
			f''(t) &= 2 \int_{\set{t}\times S^1} (\partial_\theta \partial_t^2 u, \partial_\theta u) + 2 ( \partial^2_{t \theta} u, \partial^2_{t \theta} u) \\
	&= 2\int_{\set{t}\times S^1} \abs{\partial^2_{t \theta} u}^2 + 2 \int_{\set{t}\times S^1} \abs{\partial_\theta^2 u}^2 - 2\int_{\set{t}\times S^1} (\partial_\theta^2 u, A(u)(\nabla u,\nabla u)).
		\end{split}
		\label{eqn:linwang}
	\end{equation}
	By setting $\gamma= f^{1/2}$ (see \cite{chen2019isolated}), we have
	\begin{equation*}
		2\gamma\gamma'= f' =2\int_{\set{t}\times S^1} (\partial_\theta u, \partial^2_{t\theta}u) \leq 2 \gamma \left( \int_{\set{t}\times S^1} \abs{\partial^2_{t\theta} u}^2 \right)^{1/2},
	\end{equation*}
	which implies that
	\begin{equation}\label{eqn:gammap}
		(\gamma')^2 \leq \int_{S^1} \abs{\partial^2_{t\theta} u}^2.
	\end{equation}
	Together with
	\[
		f''(t)= 2\gamma\gamma''+ 2(\gamma')^2,
	\]
	\eqref{eqn:linwang} and \eqref{eqn:gammap} imply that
	\begin{equation}\label{eqn:newlinwang}
		2\gamma \gamma'' \geq 2\int_{\set{t}\times S^1} \abs{\partial^2_\theta u}^2 - c_1 \int_{\set{t}\times S^1} \abs{\partial^2_\theta u} \abs{\nabla u}^2.
	\end{equation}

	By the Poincar\'e inequality
	\[
		\int_{\set{t}\times S^1} \abs{\partial_\theta^2 u}^2 \geq \int_{\set{t}\times S^1} \abs{\partial_\theta u}^2 = \gamma^2
	\]
	and \eqref{eqn:DT1}, we obtain from \eqref{eqn:newlinwang} that
\begin{equation*}
	\gamma''\geq \gamma - c_2 \varepsilon_1^{1/2} \gamma.
\end{equation*}
We assume that $\varepsilon_1$ is small so that $1-c_2 \varepsilon_1^{1/2}\geq \frac{9}{10}$ so that
\[
	\gamma''\geq \frac{9}{10} \gamma.
\]

Let $h$ be the solution of the ODE
\begin{equation*}
	h''=\frac{9}{10}h, \qquad h(0)=\frac{a}{\sqrt{2}}, \qquad h(T)=\frac{b}{\sqrt{2}}.
\end{equation*}
Then ODE comparison shows
\begin{equation}\label{eqn:gamma1}
	\gamma\leq h.
\end{equation}
This implies some decay of $\gamma$ along the neck. However, the decay rate is not optimal. To improve it, we would like to use \eqref{eqn:newlinwang} again. More precisely, elliptic estimate implies that for any $s\in [1,T-1]$, we have
\begin{equation}
	\max_{\set{s}\times S^1} \abs{\partial^2_\theta u} \leq c_3 \max_{[s-1,s+1]} \gamma \leq c_3 \max_{[s-1,s+1]}h\leq c_3 \tilde{c} h(s).
	\label{eqn:elliptic}
\end{equation}
Here in the last inequality above, we have used \eqref{eqn:universal} of Lemma \ref{lem:app}. 
By \eqref{eqn:universal} of Lemma \ref{lem:app} again and \eqref{eqn:difficult}, we obtain
\begin{equation*}
	\max_{\set{s}\times S^1} \abs{\partial^2_\theta u}\leq \tilde{c} C_2(L) h(s), \qquad \forall s\in [T-1,T].
\end{equation*}
The same inequality holds for $s\in [0,1]$, because of \eqref{eqn:difficult} and \eqref{eqn:toleft} of Lemma \ref{lem:app}.

With the new upper bound of $\abs{\partial_\theta^2 u}$, we derive from \eqref{eqn:newlinwang}
\begin{equation}\label{eqn:compareagain}
	\gamma''\geq \gamma - c_4 C_2(L) h^2 \qquad \forall s \in [0,T].
\end{equation}

Let $H$ be the solution of the ODE
\begin{equation*}
	H''=\frac{9}{5}H, \qquad H(0)=\frac{a^2}{2}, \qquad h(T)=\frac{b^2}{2}.
\end{equation*}
We claim that 
\begin{equation}
	h^2\leq H \qquad \text{on} \quad [0,T]. 
	\label{eqn:gamma2}
\end{equation}
In fact, 
\begin{equation*}
	(h^2)''= 2(h')^2 + \frac{9}{5} h^2 \geq \frac{9}{5}h^2,\qquad h(0)=\frac{a^2}{2},\qquad h(T)=\frac{b^2}{2}.
\end{equation*}
The claim follows from ODE comparison again. 

Combining \eqref{eqn:compareagain} and \eqref{eqn:gamma2}, we obtain 
\begin{equation*}
	\gamma''\geq \gamma - c_4 C_2(L) H\qquad \text{on} \qquad [0,T].
\end{equation*}
Hence, if $c_5=\frac{5}{4}c_4$, we have
\begin{equation*}
	(\gamma+ c_5 C_2(L) H)''\geq (\gamma+ c_5 C_2(L) H)\qquad \text{on} \qquad [0,T].
\end{equation*}
Moreover, the assumption \eqref{eqn:difficult} implies that 
\begin{eqnarray*}
	(\gamma+ c_5 C_2(L) H)(0) &\leq& (h + c_5 C_2(L)H)(0) \\
	&=& (\frac{a}{\sqrt{2}}+ c_5 C_2(L) \frac{a^2}{2})\\
	&\leq& \sqrt{2} a,
\end{eqnarray*}
if we require $\varepsilon_1^{1/2} c_5$ to be small.
Similarly,
\begin{eqnarray*}
	(\gamma+c_5 C_2(L)H)(T) &\leq& \sqrt{2}b. 
\end{eqnarray*}
ODE comparison again gives that 
\begin{equation*}
	\gamma\leq \gamma+c_6H \leq \frac{g}{\sqrt{2}} \qquad \text{on} \qquad [0,T]
\end{equation*}
where $g$ is the solution to
\begin{equation*}
	g''=g,\qquad g(0)=2{a},\qquad g(T)=2{b}.
\end{equation*}
\end{proof}

With Proposition \ref{prop:sharp}, we are now ready to finish the proof of Lemma \ref{lem:mild}. The growth condition \eqref{eqn:growth} implies that
\begin{equation*}
	\int_{\set{t_4}\times S^1} \abs{\nabla u}^2 \geq \int_{\set{t_2}\times S^1} \abs{\nabla u}^2.
\end{equation*}
Setting $L'=(t_4-t_2)$ and noticing that $L'\in (L,3L)$, we can derive \eqref{eqn:final} (hence finish the proof of Lemma \ref{lem:mild}) from the following corollary.

\begin{cor}
	\label{cor:sharp} Assume that $u$ satisfies all assumptions of Proposition \ref{prop:sharp} with $L\geq 2$ and $T>4L$. If
	\begin{equation}\label{eqn:upward}
		\int_{\set{L'}\times S^1} \abs{\nabla u}^2 \geq \int_{\set{0}\times S^1} \abs{\nabla u}^2
	\end{equation}
	for some $L'\in [L,3L]$, 
	then 
	\begin{equation*}
		\int_{\set{0}\times S^1} \abs{\nabla u}^2 \leq C_3(L) e^{-2T} \int_{\set{T}\times S^1} \abs{\nabla u}^2.
	\end{equation*}
\end{cor}

\begin{proof}
	Proposition \ref{prop:sharp} and \eqref{eqn:upward} imply that
	\begin{equation*}
		a\leq 2 \frac{ae^{2T}-be^T}{e^{2T}-1} e^{-L'} + 2 \frac{be^T-a}{e^{2T}-1} e^{L'}.
	\end{equation*}
	Hence,
	\begin{equation*}
		b \frac{e^{T+L'}-e^{T-L'}}{e^{2T}-1} \geq a \left( \frac{1}{2} -\frac{e^{2T-L'}- e^{L'}}{e^{2T}-1} \right).
	\end{equation*}
	Recalling that $L'\in (L,3L)$ and that $L>\log 100$, we get
	\begin{equation*}
		e^{3L} e^{-T} b \frac{1-e^{-4L}}{1-e^{-2T}} \geq a \left( \frac{1}{2} - e^{-L}\frac{1- e^{-2T+2L}}{1-e^{-2T}} \right).
	\end{equation*}
	The proof is done by taking $C_3(L)=16 e^{6L}$ and noticing that $T>8$.
\end{proof}

\section{Proof of the main theorem}\label{sec:proof}
The proof of Theorem \ref{thm:main} consists of three steps. First, by studying the relation between $\tilde{g}_i$ and $\bar{g}_i$, we give an equivalent form of the main theorem. Then, we prove a weak decay estimate by using the three circle lemma (Lemma \ref{lem:3c}) and its generalization (Lemma \ref{lem:4d}). Finally, we use the results in Section \ref{sub:optimal} to improve the weak decay into a sharp one, which is exactly our main theorem.

\subsection{An equivalent form of the main theorem}\label{sub:equivalent}
As before,
\begin{equation*}
	\Omega= B(0,\delta)\setminus \bigcup_{j=1,\cdots,l} B(x_i^{(j)},\delta^{-1}\lambda_i^{(j)}).
\end{equation*}
Recall that we have defined two metrics $\bar{g}_i$ and $\tilde{g}_i$ on $\Omega$. The following lemma compares them.
\begin{lem}
	\label{lem:gmgc} There is a constant $C>0$ such that for any $z\in \Omega$,
	\begin{equation*}
		\frac{1}{C}\leq \frac{\bar{g}_i}{ e^{2d(z)} \tilde{g}_i} \leq C,
	\end{equation*}
	where $d(z)$ is the distance from $z$ to $\partial \Omega$ with respect to $\bar{g}_i$.
\end{lem}

Before the proof, we notice that it implies that the following theorem is equivalent to Theorem \ref{thm:main}.
\begin{thm}
	\label{thm:main1p}
	Suppose that $u_i$ is a sequence of harmonic maps satisfying (U1) and (U2) and that $\Omega_i$ is a generalized neck domain, then there is some constant $C$ such that
	\begin{equation}
		\label{eqn:main1p}
		\abs{\nabla u_i}_{\bar{g}_i} \leq C e^{-d(z)} \qquad \text{on} \quad \Omega_i.
	\end{equation}
\end{thm}

By definition(see \eqref{eqn:bargi} and \eqref{eqn:tildegi}), for the proof of Lemma \ref{lem:gmgc}, it suffices to show that there exists $C>0$ such that
\begin{equation}\label{eqn:deal}
	\frac{1}{C}w(z) \le e^{2d(z)} \omega(z) \leq C w(z), \qquad \forall z\in \Omega
\end{equation}
where
\begin{equation}\label{eqn:defomega}
	\omega(z)= 1+ \sum_{j=1\ldots l} \frac{(\lambda_i^{(j)})^2}{ \abs{z- x_i^{(j)}}^4}.
\end{equation}

For the proof, it is essential to understand the meaning of $d(z)$ and the contribution of each term in the sum of \eqref{eqn:defomega}. Figure \ref{fig:2} illustrates an example and it is helpful in understanding the proof that follows. Here $R_1,R_2,R_3$ represent three real bubbles; $G_1,G_2$ two ghost bubbles; for a point $z\in \Omega$, we denote by $p_1,\cdots,p_4$ the four paths from $z$ to the components of $\partial \Omega$. In this case, $d(z)$ is going to be the minimal length of $p_1,\cdots,p_4$.

\begin{figure}[h]
	\centering
	\includegraphics[width=4cm]{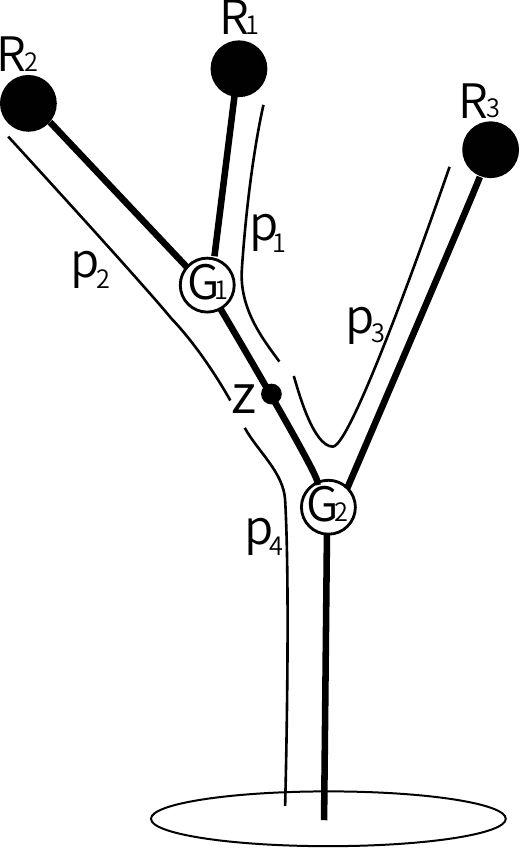}
	\caption{Distance function to the boundary}
	\label{fig:2}
\end{figure}

In the following proof, we write $a_i\sim b_i$ if there is $C>0$ such that $\frac{1}{C}a_i\leq b_i\leq Ca_i$.

\begin{proof}
By taking logarithm, it suffices to show
\begin{equation}
	\abs{-\frac{1}{2}\log \frac{\omega(z)}{w(z)} -d(z)}\leq C\qquad \forall z\in \Omega.
	\label{eqn:deallog}
\end{equation}

The rest of the proof deals with two cases separately: $z$ lies in a ghost bubble domain, or in a simple neck domain.

{\bf Case 1.} Assume that $z$ is in the ghost bubble domain
\[
	B(c_i, 2 \sigma_i) \setminus \bigcup_{y\in \mathcal C} B(y,\delta \sigma_i).
\]
Take any $z'\in \partial B(c_i, 2\sigma_i)$. We claim that
\[
	\abs{\left( -\frac{1}{2}\log \frac{\omega(z')}{w(z')} -d(z') \right)- \left( -\frac{1}{2}\log \frac{\omega(z)}{w(z)} -d(z) \right)}\leq C.
\]
Hence, it suffices to prove \eqref{eqn:deallog} for $z\in \partial B(c_i, 2\sigma_i)$, which is Case 2.
To show the claim, recall (W1) in the definition of $w$ on ghost bubble domain, which implies that
\[
	w(z)\sim w(z') \quad \text{and} \quad \abs{d(z)-d(z')}\leq C.
\]
Next, we study the difference between $\omega(z)$ and $\omega(z')$. For each $j$, the bubble $(x_i^{(j)},\lambda_i^{(j)})$ is either on top of $(c_i,\sigma_i)$, or is separate from $(c_i,\sigma_i)$. In the first case, we have
\begin{equation*}
	\lim_{i\to \infty}\frac{x_i^{(j)}-c_i}{\sigma_i} \in \mathcal C.
\end{equation*}
Hence, for $i$ sufficiently large,
\[
	\abs{z-x_i^{(j)}},\abs{z'-x_i^{(j)}} \in [\delta/2\sigma_i,4\sigma_i],
\]
which implies that 
\begin{equation}
	\label{eqn:similarj}
	\frac{(\lambda_i^{(j)})^2}{\abs{z-x_i^{(j)}}^4} \sim \frac{(\lambda_i^{(j)})^2}{\abs{z'-x_i^{(j)}}^4}.
\end{equation}
In the second case, we have
\begin{equation*}
	\lim_{i\to \infty} \frac{\abs{c_i- x_i^{(j)}}}{\sigma_i} = \infty.
\end{equation*}
Together with the fact that
\[
	\abs{c_i-z}, \abs{c_i-z'}\leq 2\sigma_i,
\]
we obtain \eqref{eqn:similarj} again.
In summary, we have $\omega(z)\sim \omega(z')$ and our claim is proved.

{\bf Case 2.} Assume that $z$ is in a simple neck domain
\[
	B(c_i,\delta \sigma_i) \setminus B(c_i,\lambda_i).
\]
First, we study the distance from $z$ to $\partial B(0,\delta)$ with respect to $\bar{g}_i$. In general, the path from $z$ to $\partial B(0,\delta)$ may pass several (or no) ghost bubble domains. For simplicity, we assume that there is only one ghost bubble domain. This is a situation illustrated by $p_4$ in Figure \ref{fig:2}. In this case, the ghost bubble $G_2$ is represented by a sequence $(x_i,\sigma_i)$ and
\begin{equation*}
	\lim_{i\to \infty}\frac{c_i-x_i}{\sigma_i}
\end{equation*}
is the concentration point at which the real bubbles $R_1$ and $R_2$ hide.
Since the diameter of the ghost bubble domain measured by $\bar{g}_i$ is bounded, we have
\begin{equation}\label{eqn:case2}
	\begin{split}
	d(z,\partial B(0,\delta)) &= d(z, \partial B(c_i,\delta \sigma_i)) + C + d(\partial B(x_i, \sigma_i), \partial B(0,\delta)) \\
	&= \log \frac{\sigma_i}{ \abs{z-c_i}} + \log \frac{1}{\sigma_i} + C\\
	&= -\log \abs{z-c_i} + C.
	\end{split}
\end{equation}

Next, we study the distance from $z$ to $\partial B(x_i^{(j)},\delta^{-1} \lambda_i^{(j)})$ for $j=1,\cdots,l$. There are two possibilities:

{\bf Case A.}  The real bubble $(x_i^{(j)}, \lambda_i^{(j)})$ sits 'on top of' the neck containing $z$, in the sense that 
\begin{equation}
	\lim_{i\to \infty} \frac{x_i^{(j)}-c_i}{ \abs{z-c_i}} < \infty;
	\label{eqn:2A}
\end{equation}

{\bf Case B.} The real bubble is separate from the neck, in the sense that
\begin{equation}
	\lim_{i\to \infty} \frac{x_i^{(j)}-c_i}{ \abs{z-c_i}} = \infty.
	\label{eqn:2B}
\end{equation}

In Figure \ref{fig:2}, the real bubbles $R_1$ and $R_2$ are Case A and the bubble $R_3$ is Case B. For case A, we assume again that the path from $z$ to $\partial B(x_i^{(j)},\lambda_i^{(j)})$ passes only one ghost bubble (see $G_1$ in Figure \ref{fig:2}), $(c_i,\lambda_i)$, so that
\[
	\lim_{i\to \infty} \frac{x_i^{(j)}-c_i}{\lambda_i} 
\]
is where $(x_i^{(j)},\lambda_i^{(j)})$ concentrates. As in Case 1, we have
\begin{equation}
	\begin{split}
		d(z,\partial B(x_i^{(j)}, \delta^{-1} \lambda_i^{(j)})) &= d(z, \partial B(c_i,2\lambda_i)) + C + d(\partial B(x_i^{(j)}, \delta \lambda_i), \partial B(x_i^{(j)},\delta^{-1} \lambda_i^{(j)})) \\
	&= \log \frac{ \abs{z-c_i}}{\lambda_i} + \log \frac{\lambda_i}{\lambda_i^{(j)}} + C\\
	&= -\log \frac{\lambda_i^{(j)}}{\abs{z-c_i}} + C.
	\end{split}
	\label{eqn:case2A}
\end{equation}

For Case B, there is a ghost bubble (see $G_2$ in Figure \ref{fig:2}) $(c'_i,\sigma'_i)$ such that
\[
	\lim_{i\to \infty} \frac{z-c'_i}{\sigma'_i} \ne \lim_{i\to \infty} \frac{x_i^{(j)}-c'_i}{ \sigma'_i}.
\]
This case is illustrated by $p_3$ in Figure \ref{fig:2}. In general, the path from $z$ to $G_2$ and from $G_2$ to $R_3$ may pass more ghost bubble domains. However, the proof remains the same by similar argument above. In this case,
\begin{equation}
	\begin{split}
		d(z,\partial B(x_i^{(j)}, \delta^{-1} \lambda_i^{(j)})) &= d(z, \partial B(c_i,\delta \sigma'_i)) + C + d(\partial B(x_i^{(j)}, \delta \sigma'_i), \partial B(x_i^{(j)},\delta^{-1}\lambda_i^{(j)})) \\
	&= \log \frac{\sigma'_i}{ \abs{z-c_i}} + \log \frac{\sigma'_i}{\lambda_i^{(j)}} + C\\
	&= -\log \frac{(\sigma_i')^2}{\lambda_i^{(j)}\cdot \abs{z-c_i}} + C.
	\end{split}
	\label{eqn:case2B}
\end{equation}

We go back to the proof of \eqref{eqn:deallog} by computing
\begin{equation}\label{eqn:dealdone}
	-\frac{1}{2}\log \frac{\omega(z)}{w(z)} =- \frac{1}{2}\log \left( \abs{z-c_i}^2 + \sum_j \frac{(\lambda_i^{(j)})^2}{ \abs{z-x_i^{(j)}}^2} \cdot \frac{\abs{z-c_i}^2}{\abs{z-x_i^{(j)}}^2} \right).
\end{equation}
Here in the parenthesis, it is the sum of $l+1$ positive terms. If we denote the largest one by $M$, we have
\begin{equation*}
	\abs{-\frac{1}{2}\log \frac{\omega(z)}{w(z)} - (-\frac{1}{2}\log M)}\leq \tilde{c}.
\end{equation*}
There are exactly $l+1$ boundary components of $\partial \Omega$. We will show that the minus logarithm of each positive term in the parenthesis is (up to a constant) the distance from $z$ to a boundary component.

First, by \eqref{eqn:case2}, the first term in the parenthesis correpsonds to the distance from $z$ to $\partial B(0,\delta)$.

For $j=1,\cdots,l$, in Case A, we have
\begin{equation*}
	\abs{z-x_i^{(j)}} \sim \abs{z-c_i},
\end{equation*}
which implies that 
\begin{equation*}
	\frac{(\lambda_i^{(j)})^2}{ \abs{z-x_i^{(j)}}^2} \cdot \frac{\abs{z-c_i}^2}{\abs{z-x_i^{(j)}}^2} \sim \frac{(\lambda_i^{(j)})^2}{ \abs{z-c_i}^2}.
\end{equation*}
It is related to the distance from $z$ to $\partial B(x^{(j)}_i, \delta^{-1} \lambda_i^{(j)})$ by \eqref{eqn:case2A}.
In case B, we have
\begin{equation*}
	\abs{z-x_i^{(j)}}\sim \sigma'_i,
\end{equation*}
which implies that
\begin{equation*}
	\frac{(\lambda_i^{(j)})^2}{ \abs{z-x_i^{(j)}}^2} \cdot \frac{\abs{z-c_i}^2}{\abs{z-x_i^{(j)}}^2} \sim \frac{(\lambda_i^{(j)})^2 \abs{z-c_i}^2}{ (\sigma'_i)^4}.
\end{equation*}
It is related to the distance from $z$ to $\partial B(x^{(j)}_i, \delta^{-1} \lambda_i^{(j)})$ by \eqref{eqn:case2B}.
\end{proof}

\subsection{A weak decay estimate}\label{sub:weak}
The aim of this subsection is to prove the following inequality that is weaker than \eqref{eqn:main1p},
\begin{equation}\label{eqn:weaker}
	\abs{\nabla u_i}_{\bar{g}_i}\leq C e^{ -\alpha d(z)} \qquad \text{on} \quad \Omega_i
\end{equation}
for some $\alpha\in (0,1)$.

Let $m_1$ be the total number of ghost bubbles and $m_2$ be the maximal number of boundary components of the ghost bubble domains. Setting $\beta=\frac{1}{2m_2}$, a constant $L$ is determined by Lemma \ref{lem:3c} and Lemma \ref{lem:4d} (with $\eta=e^{-L}$). As $i$ goes to infinity, so do the lengths of simple neck domains. Assume without loss of generality that these lengths are integer multiples of $L$. Then the generalized neck domain becomes the union of many cylinder pieces
\begin{equation*}
	W=[0,L]\times S^1
\end{equation*}
and one ghost bubble piece(domain)
\begin{equation*}
	W=B(x_i,2\sigma_i)\setminus \bigcup_{y\in \mathcal C} B(c_i(y),\delta \sigma_i). 
\end{equation*}
In either case, we write $E(W)$ for the integral $\int_W \abs{\nabla u}_{\bar{g}_i}^2$.

Since the diameter of each piece (w.r.t. $\bar{g}_i$) is bounded, the weak decay estimate \eqref{eqn:weaker} follows from the next lemma. 

\begin{lem}
	\label{lem:rough}For any piece $W$ as above, there are $s$ pieces $W_1,\cdots,W_s$ such that

	(i) $W=W_1$;

	(ii) $W_s$ touches one component of $\partial \Omega$;

	(iii) $\bigcup_{j=1,\cdots,s}W_j$ is connected;

	(iv) The number of index $k$ from $2$ to $s$ {\bf not} satisfying
	\begin{equation*}
		2E(W_{k-1})\leq E(W_k)
	\end{equation*}
	is bounded by $4m_1$;

	(v) For any $j=1,\cdots,s$,
	\begin{equation*}
		E(W_j)\leq C 2^{-j}.
	\end{equation*}
\end{lem}

\begin{figure}[h]
	\centering
	\includegraphics{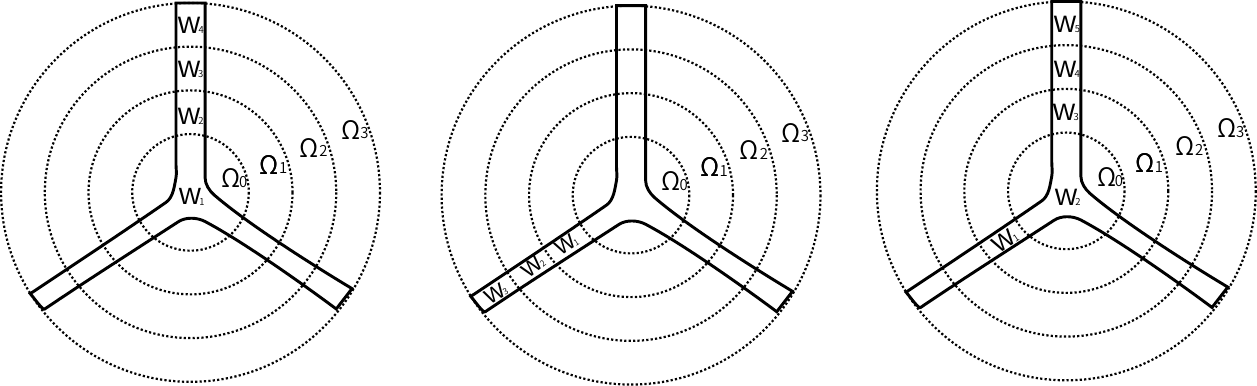}
	\caption{From left to right: Case 1, Case 2A and Case 2B}
	\label{fig:3}
\end{figure}

\begin{proof}
The proof is by induction. Set $W_1=W$.

There are three cases depending on the position of $W$.

Case 1: $W$ is a ghost bubble domain(see Case 1 of Figure \ref{fig:3}). By Lemma \ref{lem:4d},
\begin{equation*}
	\int_{\Omega_2} \abs{\nabla u}^2 \leq \beta \int_{\Omega_3\setminus \Omega_2} \abs{\nabla u}^2.
\end{equation*}
Since $\beta=\frac{1}{2m_2}$, there is a component of $\Omega_3\setminus \Omega_2$, which is a cylinder piece denoted by $W_4$, such that
\begin{equation*}
	\int_{\Omega_2} \abs{\nabla u}^2 \leq \frac{1}{2} E(W_4).
\end{equation*}
Let $W_2$ and $W_3$ be the two pieces connecting $W$ and $W_4$. Notice that $W_3$ is a cylinder and we have $E(W_4)\geq 2 E(W_3)$.

Case 2: If $W$ is a cylinder next to a ghost bubble domain, we use Lemma \ref{lem:4d} as above to get a component $W'$ of $\Omega_3\setminus \Omega_2$ satisfying
\begin{equation*}
	\int_{\Omega_2} \abs{\nabla u}^2 \leq \frac{1}{2} E(W').
\end{equation*}
If $W'$ and $W$ is in the same component of $\Omega_3\setminus \Omega_0$ (see Case 2A of Figure \ref{fig:3}), then we set $W'=W_3$ and let $W_2$ be the piece between $W_3$ and $W_1$. Notice that $W_2$ is a cylinder and that $E(W_3)\geq 2 E(W_2)$.

If $W'$ and $W$ are not in the same component of $\Omega_3\setminus \Omega_0$ (see Case 2B of Figure \ref{fig:3}), then we set $W_2$ to be the ghost bubble domain, $W_5=W'$ and $W_3,W_4$ be the two pieces between $W_2$ and $W_5$.

In this case, it is also true that $W_4$ and $W_5$ are cylinders and $E(W_5)\geq 2 E(W_4)$.

Case 3: $W$ is a cylinder and the two adjacent pieces are also cylinders. By Lemma \ref{lem:3c}, since $\beta\leq \frac{1}{4}$, there is at least one adjacent piece, which we denote by $W_2$ satisfying $E(W_2)\geq 2 E(W_1)$.

Assume that we have found $W_1,\cdots,W_j$ such that

(H1) $W_{k-1}$ is adjacent to $W_k$ for any $k=2,\cdots,j$.

(H2) $W_{j-1}$ and $W_j$ are both cylinders satisfying $E(W_j)\geq 2E(W_{j-1})$;

(H3) the number of index $k$ from $2$ to $j$ not satisfying
\begin{equation*}
	2E(W_{k-1})\leq E(W_k)
\end{equation*}
is bounded by $4m_1$.

If $W_j$ touches the boundary of $\partial \Omega$, then the induction is complete. To conclude the proof, it remains to justify (iv) and (v). Notice that (iv) is just (H3) and (v) follows from (iv) and the fact that for any piece $W'$ touching the boundary, we have $E(W')\leq C\varepsilon_1$.

If $W_j$ does not touch the boundary, we define $W_{j+1}$ as follows.
By (H2), $W_j$ is a cylinder. Hence, there are two adjacent pieces and one of them is $W_{j-1}$. Denote the other by $W'$.

If $W'$ is a cylinder, then by Lemma \ref{lem:3c}, we conclude that $E(W')\geq 2E(W_j)$ and denote $W'$ by $W_{j+1}$. 

\begin{figure}[h]
	\centering
	\includegraphics[width=6cm]{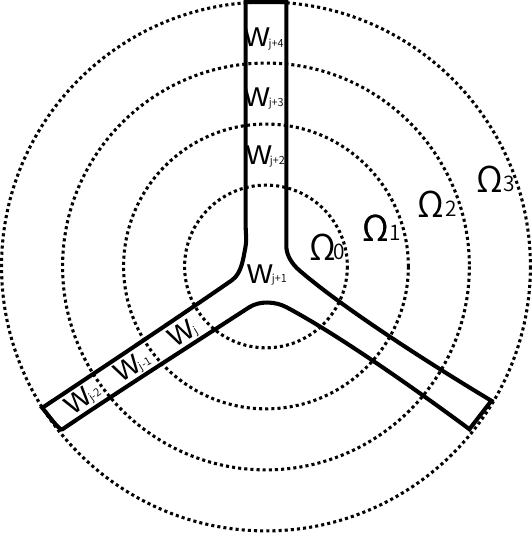}
	\caption{passing the ghost bubble domain}
	\label{fig:4}
\end{figure}

If $W'$ is a ghost bubble domain(see Figure \ref{fig:4}), we choose $W''$ to be the component in $\Omega_3\setminus \Omega_2$ satisfying
\begin{equation}\label{eqn:wpp}
	\int_{\Omega_2} \abs{\nabla u}^2 \leq \frac{1}{2} E(W'').
\end{equation}
Notice that $W''$ and $W_j$ can not be in the same component of $\Omega_3\setminus \Omega_0$, otherwise $W''$ would be $W_{j-2}$, which contradicts \eqref{eqn:wpp}. Then we set $W_{j+1}=W'$, $W_{j+4}=W''$ and let $W_{j+2}, W_{j+3}$ be the two pieces between $W'$ and $W''$.

It is easy to check that the induction hypothesis (H1)-(H3) hold. We repeat the induction construction until the proof is done.
\end{proof}

\subsection{The optimal decay estimate}\label{sub:optimalcay}

The aim of this subsection is to prove \eqref{eqn:main1p}. The proof is based on Lemma \ref{lem:rough} and Lemma \ref{lem:mild}. 

Recall that in the proof of Lemma \ref{lem:rough}, $\Omega$ is decomposed into cylinder pieces of length $L$ and ghost bubble piece. Our first step in the proof of Theorem \ref{thm:main1p} is to show that it suffices to prove \eqref{eqn:main1p} for $z\in \Omega$ satisfying
\[
d(z,W)>15L
\]
for any ghost bubble piece $W$. Here $d(z,W)$ is measured with respect to $\bar{g}_i$.

Assume this is true and let $\tilde{z}$ be any point satisfying $d(\tilde{z},W)\leq 15L$ for some ghost bubble domain $W$. With $W=\Omega_0$ in mind, we recall the definition of $\Omega_1,\Omega_2,\cdots$ in \eqref{eqn:omegaj}. By Lemma \ref{lem:4d} and Lemma \ref{lem:3c}, we have
\begin{equation*}
	E(\Omega_2)\leq \frac{1}{2} E(\Omega_3\setminus \Omega_2) \leq \frac{1}{4} E(\Omega_4\setminus \Omega_3) \leq \cdots.
\end{equation*}
Together with elliptic estimates, the above inequality implies that
\begin{equation*}
	\abs{\nabla u}_{\bar{g}_i}^2(\tilde{z})\leq C \int_{\Omega_{16}}\abs{\nabla u}_{\bar{g}_i}^2 \leq C \int_{\Omega_{17}\setminus \Omega_{16}} \abs{\nabla u}_{\bar{g}_i}^2.
\end{equation*}
By our assumption, for any $z'\in \Omega_{17}\setminus \Omega_{16}$, we have $d(z',W)>15L$ and hence
\begin{equation*}
	\abs{\nabla u}_{\bar{g}_i}^2(\tilde{z})\leq  C \sup_{z'\in \Omega_{17}\setminus \Omega_{16}}e^{-2d(z')} \leq C e^{-2d(\tilde{z})},
\end{equation*}
because $d(\tilde{z},z')\leq C$.

Hence, for the rest of the proof we assume that $d(z,W)>15L$ for any ghost bubble piece. Let $W_z$ be the cylinder piece containing $z$ and $W_-$ and $W_+$ be the two adjacent pieces. By elliptic estimates, we have
\begin{equation*}
	\abs{\nabla u}^2_{\bar{g}_i}(z) \leq C \left( E(W_-)+ E(W_z) + E(W_+) \right). 
\end{equation*}
Therefore, it suffices to show
\begin{equation*}
	E(W)\leq C e^{-2d(W,\partial \Omega)}
\end{equation*}
for any cylinder piece $W$ whose distance to any ghost bubble piece is larger than $12L$.

For this $W$, Lemma \ref{lem:rough} gives a sequence $W_1,\cdots,W_s$. For simplicy, we assume only one of them, say $W_l$, is a ghost bubble piece. The proof of Lemma \ref{lem:rough} shows that
\begin{equation*}
	2E(W_{k-1})\leq E(W_k)
\end{equation*}
for $k=2,3,\cdots,l-1$ and $k=l+3,l+4,\cdots,s$ (see Figure \ref{fig:4}). Hence, we can apply Lemma \ref{lem:mild} to see
\begin{equation*}
	E(W_1)\leq C(L) e^{-2l L} E(W_{l-1})
\end{equation*}
and
\begin{equation*}
	E(W_{l+3})\leq C(L) e^{-2(s-l) L} E(W_{s}).
\end{equation*}
Moreover, we have $E(W_{l+3})\geq E(W_{l-1})$ as a consequence of Lemma \ref{lem:4d}. This finishes the proof of Theorem \ref{thm:main1p}.

\section{An application}\label{sec:app}
In this section, we prove Theorem \ref{thm:better}. 
To simplify the notations, we assume that there is only one energy concentration point $p\in \Sigma$ and that there are several real bubbles concentrated at $p$ and these bubbles are all separated. Hence, the real bubbles and the weak limit are connected with only one generalized neck domain. More precisely, take a conformal coordinate centered at $p$ and assume that the real bubbles are 
\[
\mathcal B_j: \quad (x_i^{(j)},\lambda_i^{(j)}), \qquad j=1,\cdots,l.
\]
By setting 
\begin{equation*}
	c_i=\frac{1}{l} \sum_{j=1}^l x_i^{(j)},
\end{equation*}
for some small $\delta_0>0$, the generalized neck domain is
\begin{equation*}
	\Omega_i= B(c_i,\delta_0) \setminus \bigcup_{j=1,\cdots,l} B(x_i^{(j)},\delta_0^{-1}\lambda_i^{(j)}).
\end{equation*}
We choose this $\delta_0$ to be small so that all results proved in previous sections hold. In what follows, we shall need another parameter $\delta\in (0,\delta_0)$ and set
\begin{equation}\label{eqn:ooo}
	\Omega_i(\delta)= B(c_i,\delta) \setminus \bigcup_{j=1,\cdots,l} B(x_i^{(j)},\delta^{-1}\lambda_i^{(j)}).
\end{equation}

The outline of the proof is the same as in \cite{yin2019higher}, which we recall below. 

\subsection{Outline of proof}
Let $u_i$ be the sequence in Theorem \ref{thm:old}. We shall define a sequence of conformal metrics $g_i$ on $\Sigma$. While $NI(u_i)$ is conformally invariant, the eigenvalues and eigenfunctions of the operator $J_{u_i}$ do depend on $g_i$. By taking a subsequence if necessary, we assume
\[
m=\lim_{i\to \infty} NI(u_i).
\]
Suppose that $\beta_{i,1},\cdots,\beta_{i,m}$ are the nonpositive eigenvalues (counting multiplicities) of $J_{u_i}$ and that $v_{i,1},\cdots,v_{i,m}$ are the corresponding eigenfunctions, i.e.
\begin{equation}
	J_{u_i}(v_{i,k})=\beta_{i,k} v_{i,k}\qquad k=1,\cdots,m.
	\label{eqn:eigen}
\end{equation}
Here $v_{i,k}$ are the sections of the pullback bundle $u_{i}^*TN$, which are normalized so that
\begin{equation}
	\int_M \langle v_{i,k}, v_{i,k'}\rangle dV_{g_i}=\delta_{k,k'}.
	\label{eqn:normalize}
\end{equation}
Notice that we have embedded $N$ into $\Real^p$ and hence $v_{i,k}$ are also regarded as $\Real^p$-valued functions that are perpendicular to the tangent space of $N$ at $u_i$.

We study the limit of $(\Sigma, g_i, u_i, \beta_{i,k}, v_{i,k})$. By our choice of $g_i$ (see below), we shall obtain a limit
\[
(\Sigma, g, u_\infty, \beta_{k}, v_{k}).
\]
Here $u_\infty$ is the weak limit of $u_i$ and $\beta_k$ and $v_k$ are the eigenvalues and eigenfunctions of $J_{u_\infty}$ respectively. For now, we do not know if they are linearly independent or not. This is a key issue that will be addressed later.

For each real bubble $\mathcal B_j (j=1,\cdots,l)$, we obtain a limit
\[
(\Real^2, g_b, \omega_j, \tilde{\beta}_k^{(j)}, \tilde{v}_k^{(j)}).
\]
For the definition of $g_b$, see \eqref{eqn:gb} in the next subsection.
Again, the eigenfunctions $\tilde{v}_k^{(j)}$ for $J_{\omega_j}$ may be linearly dependent.

To proved the desired inequality in Theorem \ref{thm:better}, we claim that
\begin{equation}
	\int_\Sigma \langle v_k,v_{k'} \rangle dV_g + \sum_{j=1}^l \int_{S^2} \langle \tilde{v}_k^{(j)}, \tilde{v}_{k'}^{(j)} \rangle dV_{g_b} = \delta_{k,k'}.
	\label{eqn:suppose}
\end{equation}
In a linear space with inner product, the dimension of the linear subspace spanned by $\alpha_1,\cdots,\alpha_m$ is the rank of the matrix
\[
\left( \langle \alpha_k, \alpha_{k'}\rangle \right)_{k,k'=1,\cdots,m}.
\]
Hence, Theorem \ref{thm:better} follows from \eqref{eqn:suppose}.

Intuitively, \eqref{eqn:suppose} is a consequence of \eqref{eqn:normalize}. For any fixed $\delta$, while the convergence on $\Sigma\setminus B(0,\delta)$ and $B(x_i^{(j)},\delta^{-1}\lambda_i^{(j)})$ is nice (see \cite{yin2019higher} for details), there is no control over the integral
\[
\int_{\Omega_i(\delta)} \langle v_{i,k}, v_{i,k'}\rangle dV_{g_i}.
\]
Therefore, the most important step in the proof of Theorem \ref{thm:better} is to show that
\[
\lim_{\delta\to 0}\lim_{i\to \infty}\int_{\Omega_i(\delta)} \langle v_{i,k}, v_{i,k'}\rangle dV_{g_i}=0,
\]
which is a consequence of

(T1) the volume of $\Omega_i(\delta)$ with respect to $g_i$ goes to zero when $\delta\to 0$;

(T2) there is some constant $C>0$ independent of $i$ such that
\[
\sup_{\Omega_i(\delta_0/16)} \abs{v_{i,k}} \leq C.
\]

\subsection{Metric $g_i$ on the generalized neck domain}
For the definition of $g_i$, we first define a metric on $\Real^2$ as follows
\begin{equation}
	\label{eqn:gb}
g_b=f(r) (dr^2 + r^2 d\theta^2)
\end{equation}
where $(r,\theta)$ is the polar coordinates and 
\[
f(r)=\left\{
\begin{array}[]{ll}
	\left( \frac{1}{1+r^2} \right)^2 & r\leq 1 \\
	\frac{1}{r^4} & r>2.
\end{array}
\right.
\]
This is supposed to be the limit of $g_i$ on each real bubble domain and it is to be connected to the $\Sigma$ by $\tilde{g}_i$ defined on the generalized neck domain, which is defined as (see \eqref{eqn:tildegi})
\begin{equation*}
	\tilde{g}_i= \left( 1+ \sum_{j=1}^l \frac{(\lambda_i^{(j)})^2}{ \abs{z-x_i^{(j)}}^4}\right) dz\wedge d\bar{z}.
\end{equation*}
With a cut-off function $\varphi:[0,+\infty)\to [0,1]$ satisfying $\varphi(x)\equiv 0$ for $s\leq 1$ and $\varphi(s)\equiv 1$ for $s\geq 2$, we define
\begin{equation}
	g_i(z)=\left\{
	\begin{array}[]{ll}
		g(z) & \text{on} \quad \Sigma \setminus B(c_i,\delta_0/2)\\
		\varphi(\frac{4\abs{z-c_i}}{\delta_0}) g + (1- \varphi(\frac{4\abs{z-c_i}}{\delta_0})) \tilde{g}_i & \text{on} \quad B(c_i,\delta_0/2)\setminus B(c_i,{\delta_0/4}) \\
		\tilde{g}_i(z) & \text{on} \quad \Omega_i(\delta_0/4) \\
		\varphi(\frac{\abs{z-x_i^{(j)}} \delta_0}{2 \lambda_i^{(j)}}) \tilde{g}_i + (1- \varphi(\frac{\abs{z-x_i^{(j)}}\delta_0}{2\lambda_i^{(j)}})) (L_l^{(j)})^* g_b & \text{on} \quad B(x_i^{(j)},4\delta_0^{-1}\lambda_i^{(j)})\setminus B(x_i^{(j)},2\delta_0^{-1}\lambda_i^{(j)}) \\
		(L_i^{(j)})^* g_b & \text{on} \quad B(x_i^{(j)},2\delta_0^{-1}\lambda_i^{(j)})
	\end{array}
	\right.
	\label{eqn:metricgi}
\end{equation}
where $j=1,\cdots,l$ and $L_i^{(j)}:\Real^2\to \Real^2$ maps $z$ to $\frac{z-x_i^{(j)}}{\lambda_i^{(j)}}$.

With this definition, it is straight forward to show an analog of Lemma 5.2 in \cite{yin2019higher}, from which (T1) follows.

\begin{lem}
	\label{lem:52} For any $\delta\in (0,\delta_0)$, we have, when $i\to \infty$,

	(1) $g_i$ converges to $g$ on $\Sigma\setminus B(c_i,\delta)$;

	(2) for each $j=1,\cdots,l$, $((L_i^{(j)})^{-1})^*g_i$ converges to $g_b$ on $B(0,\delta^{-1})$;

	(3) The volume of $\Omega_i(\delta)$ with respect to $g_i$ is bounded by $C \delta^2$ for some universal constant $C>0$.
\end{lem}

As explained in Section \ref{sub:example}, $\tilde{g}_i$ is the pullback metric of some minimal embedding. In \cite{yin2019higher}, on a simple neck domain (or a cylinder), an explicit parametrization of the catenoid in $\Real^3$ was used. Here, we use the sequence $u_i$ given in \eqref{eqn:defui}. The following mean value inequality is a generalization of Lemma 5.3 of \cite{yin2019higher}.

\begin{lem}
	\label{lem:53}
	For any positive number $C_1>0$, there is $C_2$ depending on $C_1$ but not $i$ such that if a nonnegative function $w$ satisfies
	\begin{equation*}
		\triangle g_i w \geq -C_1 w, \qquad \text{on \quad} \Omega_i(\delta_0/8)
	\end{equation*}
then for sufficiently large $i$,
\[
\sup_{\Omega_i(\delta_0/16)}w \leq C_2 \int_{\Omega_i(\delta_0/8)} w dV_{g_i}.
\]
\end{lem}

\begin{proof}
	Recall that the metric $g_i$ restricted to $\Omega_i(\delta_0/8)$ is the same as $\tilde{g}_i$ and $\tilde{g}_i$ is the pullback metric by $\tilde{u}_i$ defined in \eqref{eqn:defui}. Since $\tilde{u}_i$ parametrizes a minimal surface in $\mathbb C^{l+1}$, the classical mean value inequality (see Corollary 1.16 of \cite{colding2011course}) implies that for any $y\in \tilde{u}_i(\Omega_i(\delta_0/16))$,
	\begin{equation*}
		w(y)\leq C_2 \int_{\hat{B}(y,\frac{\delta_0}{32})\cap \tilde{u}_i(\Omega_i(\delta_0/8))} w dV_\Sigma,
	\end{equation*}
	as long as we verify that 
	\begin{equation}
	\hat{B}(y,\frac{\delta_0}{32})\cap  \tilde{u}_i( \partial\Omega_i(\delta_0/8))=\emptyset.
		\label{eqn:verify}
	\end{equation}
	Here $dV_\Sigma$ is the induced metric on the image of $\tilde{u}_i$ and $\hat{B}$ is the metric ball in $\mathbb C^{l+1}$. 
	To see that \eqref{eqn:verify} holds for large $i$, we consider the limit of $\tilde{u}_i(\Omega_i(\delta))$ as a subset in $\mathbb C^{l+1}$, which by \eqref{eqn:defui} and \eqref{eqn:ooo} is
	\begin{equation*}
		\tilde{\Omega}(\delta):=\bigcup_{j=1,\cdots,l+1} \set{(z_1,\cdots,z_{l+1})|\quad \abs{z_j}\leq \delta, \quad z_k=0 \quad \text{for} \quad k\ne j}.
	\end{equation*}
	Hence, the boundary of $\tilde{\Omega}(\delta_0/8)$ is
\[
\bigcup_{j=1,\cdots,l+1} \set{(z_1,\cdots,z_{l+1})|\quad \abs{z_j}= \delta_0/8, \quad z_k=0 \quad \text{for} \quad k\ne j},
\]
whose distance to $\tilde{u}_i(\Omega_i(\delta_0/16))$ is greater than $\delta_0/32$. 
\end{proof}

With Lemma \ref{lem:53}, we may prove (T2) as Lemma 5.7 \cite{yin2019higher}. Notice that in this proof, we used Theorem \ref{thm:main} in the form that
\begin{equation*}
	\sup_{\Omega_i(\delta_0)}\norm{\nabla u_i}_{\tilde{g}_i} \leq C.
\end{equation*}

With (T1) and (T2), the rest of the proof is the same as in \cite{yin2019higher}.

\appendix
\section{Some properties of an ODE solution}

In this appendix, we show some elementary properties of the solution $g(t)$ to the ordinary differential equation with boundary values
\begin{equation*}
	g''(t)=\gamma^2 g(t),\qquad g(0)=a \quad \text{and} \quad g(T)=b.
\end{equation*}
We assume that $\gamma>1/2$ and $T>5$. They are not essential and we assume these for simplicity.

\begin{lem}
	\label{lem:app} There is a universal constant $\tilde{c}$ such that for any positive constants $a$ and $b$ with $b\geq a$, we have
	\begin{equation}
		\sup_{[1,T]}\abs{(\log g)'} \leq \tilde{c}
		\label{eqn:universal}
	\end{equation}
	and
	\begin{equation}
		a\leq \tilde{c} \inf_{[0,1]} g.
		\label{eqn:toleft}
	\end{equation}
\end{lem}

\begin{rem}
	In general, since $b$ may be very large, even a lot larger than $e^{\gamma T}a$, we can not expect an upper bound of $(\log g)'$ over $[0,T]$. The observation is that such an upper bound holds for $[1,T]$ regardless of the size of $a$, $b$ and $T$.
\end{rem}

The proof follows from explicit computation, since we have the following formula for the solution
\begin{equation}\label{eqn:g}
	g(t)=\frac{a e^{2\gamma T}- be^{\gamma T}}{e^{2\gamma T}-1} e^{-\gamma t} + \frac{b e^{\gamma T}-a}{e^{2\gamma T}-1} e^{\gamma t}.
\end{equation}
Direct computation shows
\begin{equation}\label{eqn:gp}
	g'(t)=- \gamma \frac{a e^{2\gamma T}- be^{\gamma T}}{e^{2\gamma T}-1} e^{-\gamma t} + \gamma \frac{b e^{\gamma T}-a}{e^{2\gamma T}-1} e^{\gamma t}
\end{equation}
and
\begin{equation}\label{eqn:loggp}
	\frac{g'(t)}{g(t)}= \gamma \frac{(be^{\gamma T}-a) e^{\gamma t} - (a e^{2\gamma T}-b e^{\gamma T}) e^{-\gamma t}}{(be^{\gamma T}-a) e^{\gamma t} + (a e^{2\gamma T}-b e^{\gamma T}) e^{-\gamma t}}.
\end{equation}
Taking one more derivative, we get
\begin{equation}\label{eqn:loggdoublep}
	\left( \frac{g'}{g} \right)' = \gamma^2 \frac{4 (be^{\gamma T}-a)(a e^{2\gamma T}-b e^{\gamma T})}{((be^{\gamma T}-a) e^{\gamma t} + (a e^{2\gamma T}-b e^{\gamma T}) e^{-\gamma t})^2}.	
\end{equation}

(i) If $a$ and $b$ are comparable in the sense that $a\leq b\leq e^{\gamma T}a$, the lemma holds with $\tilde{c}=\gamma$. To see this, we notice that in this case
\begin{equation*}
	{a e^{2\gamma T}- be^{\gamma T}}, {b e^{\gamma T}-a}\geq 0.
\end{equation*}
Hence, by \eqref{eqn:loggp}, we have $\abs{(\log g)'}\leq \gamma$ for all $t\in [0,T]$, from which both \eqref{eqn:universal} and \eqref{eqn:toleft} follow.

(ii) If $b\geq e^{\gamma T}a$, \eqref{eqn:gp} implies that $g'\geq 0$.
Hence, $g$ is increasing and \eqref{eqn:toleft} follows.

 Moreover, \eqref{eqn:loggdoublep} implies that $(\log g)''\leq 0$.
 Together with the observation
\begin{equation*}
	(\log g)' (T) = \gamma \frac{b(e^{2\gamma T}+1)-2ae^{\gamma T}}{b(e^{2\gamma T}-1)} \in (0,2\gamma),
\end{equation*}
 it suffices to bound $(\log g)'(1)$, which we compute
\begin{eqnarray*}
	(\log g)'(1) &=& \gamma \frac{(b e^{\gamma T}-a) e^\gamma - (a e^{2\gamma T}-b e^{\gamma T})e^{-\gamma}}{(b e^{\gamma T}-a) e^\gamma + (a e^{2\gamma T}-b e^{\gamma T})e^{-\gamma}} \\
	&=& \gamma \frac{b e^{\gamma(T+1)} + b e^{\gamma(T-1)} - a e^\gamma - ae^{\gamma (2T-1)}}{b e^{\gamma(T+1)} - b e^{\gamma(T-1)} - a e^\gamma + a e^{\gamma (2T-1)}}\\
	&\leq & \gamma \frac{b e^{\gamma(T+1)} + b e^{\gamma(T-1)}}{b e^{\gamma(T+1)} - b e^{\gamma(T-1)} - a e^\gamma}.
\end{eqnarray*}
Finally, we notice that
\begin{equation*}
	a e^\gamma \leq b e^{-\gamma T + \gamma},
\end{equation*}
which implies that
\begin{equation*}
	(\log g)'(1) \leq \frac{2 \gamma}{1- e^{-2\gamma}- e^{-2\gamma T}} \leq 4\gamma.
\end{equation*}
This concludes the proof of the lemma.

\bibliographystyle{alpha}

\bibliography{foo}
\end{document}